\documentclass[11pt,a4paper]{amsart}
\usepackage[utf8]{inputenc}
\usepackage[english]{babel}
\usepackage{amsmath}
\usepackage{dsfont}
\usepackage{graphicx}
\usepackage{color}
\usepackage[all]{xy}
\usepackage{indentfirst}
\theoremstyle{plain}
\newtheorem{theorem}{Theorem}[section]
\newtheorem{exmp}[theorem]{Example}
\newtheorem{remark}[theorem]{Remark}

\newtheorem{thm-def}[theorem]{Theorem/Definition}
\newtheorem{definition}[theorem]{Definition}
\newtheorem{proposition}[theorem]{Proposition}
\newtheorem{lemma}[theorem]{Lemma}
\newtheorem{corollary}[theorem]{Corollary}

\def\nota#1{\relax}

\def\Projan{\operatorname{Projan}}%{\mathop{\rm Projan}}
\let\:=\colon

\begin{document}\baselineskip=1.25em

%\author[T. Gaffney]{Terence Gaffney}
%\thanks{T.~Gaffney was partially supported by PVE-CNPq Proc. 401565/2014-9}
 \address{T. Gaffney, Department of Mathematics\\
  Northeastern University\\
  Boston, MA 02215\\
  t.gaffney@neu.edu}
  
  %\author[N. Grulha]{Nivaldo G. Grulha, Jr.}
 % \thanks{N.~Grulha was partially supported by FAPESP Proc. 2015/16746-7 and CNPq Proc. 474289/2013-3 and 303641/2013-4 }
  \address {N. Grulha\\
    Instituto de Ci\^encias Matem\'aticas e de Computa\c{c}\~ao - USP\\
   Av. Trabalhador s\~ao-carlense, 400 - Centro\\
  CEP: 13566-590 - S\~ao Carlos - SP, Brazil\\
  njunior@icmc.usp.br}
  
  %\author[M. A. Ruas]{Maria Aparecida Ruas}
  %\thanks{M.A.~Ruas was partially supported by FAPESP Proc. 2014/00304-2 and CNPq Proc. 305651/2011-0}
\address {M. Ruas\\
  Instituto de Ci\^encias Matem\'aticas e de Computa\c{c}\~ao - USP\\
 Av. Trabalhador s\~ao-carlense, 400 - Centro\\
CEP: 13566-590 - S\~ao Carlos - SP, Brazil\\
maasruas@icmc.usp.br}

\title[The local Euler obstruction and topology]{The local Euler obstruction and  topology of the stabilization of associated determinantal varieties}
\author{Terence Gaffney,  Nivaldo G.\  Grulha Jr.\  and Maria A.\  S.\  Ruas}
\let\thefootnote\relax
%\address{Terence Gaffney, Department of Mathematics,\newline Northeastern University, Boston, MA 02215, USA.}
%\email{t.gaffney@northeastern.edu}

%\address{ Nivaldo G. Grulha Jr and Maria Aparecida Soares Ruas, \newline Universidade de S\~{a}o Paulo, Instituto de Ci\^{e}ncias Matem\'{a}ticas e de Computa\c{c}\~{a}o, Departamento de Matem\'{a}tica.\\  Av. Trabalhador S\~{a}o-carlense, 400 - Centro, Caixa Postal: 668 - CEP: 13560-970 - S\~{a}o Carlos - SP - Brazil.}
%\email{njunior@icmc.usp.br}
%\email{maasruas@icmc.usp.br}

\maketitle
	
\begin{abstract}
This work  has two complementary parts, in the first part we compute the local Euler obstruction of generic determinantal varieties and apply this result to compute the Chern--Schwartz--MacPherson class of such varieties. In the second part  we compute the Euler characteristic of the stabilization of an essentially isolated determinantal singularity (EIDS). The formula is given in terms of the local Euler obstruction and  Gaffney's $m_{d}$ multiplicity.
\end{abstract}

%\abstractname{: This work is split into two parts, in the first part we compute the local Euler obstruction of generic determinantal varieties and apply this result to compute the Chern-Schwartz-MacPherson class of such varieties. On the second part  we compute the Euler characteristic of the stabilization of an essentially isolated determinantal singularity (EIDS). The formula is given in terms of the local Euler obstruction and the Gaffney's $m_{d}$ multiplicity.} 

\section*{Introduction}

In \cite{M}  MacPherson proved the existence and %unicity 
uniqueness of Chern classes for possibly singular complex algebraic varieties. The local Euler obstruction, defined by MacPherson in that paper, was one of the main ingredients in his proof. 

The computation of the local Euler obstruction is not easy; various authors propose formulas which make the computation easier.  For instance, L\^{e}  and
Teissier provide a formula in terms of polar multiplicities \cite{LT1}.

In \cite{BLS}, Brasselet, L\^{e}  and Seade
give a Lefschetz type formula for the local Euler obstruction. The formula shows that
the local Euler obstruction, as a constructible function, satisfies the Euler condition
relative to generic linear forms.

In order to understand these ideas better,  some authors worked on some more specific situations. For example, in the special case of toric surfaces, an interesting formula for the Euler obstruction was proved by Gonzalez--Sprinberg \cite{GS}, this formula was generalized by Matsui and Takeuchi for normal toric varieties \cite{MT}.

A natural class of singular varieties to investigate the local Euler obstruction and the generalizations of the characteristic classes is the class of generic determinantal varieties (Def.\ \ref{DefGenDet}). Roughly speaking, generic determinantal varieties are sets of matrices with a given upper bound on their ranks. Their significance comes, for instance, from the fact that many examples in algebraic geometry are of this type, such as the Segre embedding of a product of two projective spaces. Independently, in recent work \cite{Z}, Zhang computed the Chern-Mather-MacPherson Class of projectivized determinantal varieties, in terms of the trace of certain matrices associated with the push forward of the MacPherson-Schwartz class of the Tjurina transform of the singularity.

In the first section we prove a surprising formula that allow us to compute the local Euler obstruction of generic determinantal varieties using only Newton binomials. Using this formula we also compute the Chern--Schwartz--MacPherson classes of such varieties.

%In the second section we compute the Euler characteristic of the stabilization of an essentially isolated determinantal singularity (EIDS, Def.\ \ref{EIDS}). The formula is given in terms of the local Euler obstruction and the $m_{d}$ multiplicity defined by Gaffney in \cite{G-Top} for the study of isolated complete intersection singularities (ICIS), and for isolated singularities whose versal deformation have a smooth base in \cite{Gaff1}. Our results of the second section are mainly based on \cite{GRa} and \cite{GR}. 

In the second section we apply the results of the first section to the study of essentially isolated determinantal singularity (EIDS, Def.\ \ref{EIDS}). These singularities are the pullbacks of the generic determinantal singularities by maps which are generic except at the origin. We first compute the Euler characteristic of the EIDS. The formula, which appears in Theorem \ref{stabilization}, is given in terms of the local Euler obstruction and the $m_{d}$ multiplicity defined by Gaffney in \cite{G-Top} for the study of isolated complete intersection singularities (ICIS), and for isolated singularities whose versal deformation have a smooth base in \cite{Gaff1}. By imposing conditions on $X$, we can ensure that $X$ is such a good approximation to the generic determinantal variety, that all of the terms of the formula come from the generic determinantal variety, except for the ICIS contained in $X$, formed from the points where the rank of the presentation matrix is $0$. This is Proposition \ref{stabilizationEIDS}. 

In fact, the original motivation for the paper came from two sources; it was noted in earlier work on {\cite{GR}} that for generic maps the Euler obstruction for generic determinantal singularities should be closely related to the Euler obstruction of the pullback. Further, the Euler obstruction of the generic determinantal singularities appeared in the formula for the Euler characteristic of the stabilization of an EIDS for sufficiently generic sections. 

So the computation of the Euler obstruction for generic determinantal singularities became important.  As mentioned above, Theorem \ref{stabilization} gives the connection between the Euler characteristic of the stabilization and the Euler obstructions of the strata, while in the case where the map defining $X$ is sufficiently generic,  Proposition \ref{stabilizationEIDS} connects the Euler characteristic of the stabilization directly with the Euler obstruction of the generic singularity.

Studying the relation between the Euler obstruction of a determinantal variety and the  Euler obstruction of a generic determinantal variety, raises the question of the functoriality of the Euler obstruction. In turn this motivates the definition of the Euler obstruction of a module, (Definition \ref{Genob})  as the Euler obstruction of $F^*(JM(\Sigma^t))$ appears in Theorem \ref{theo20}. 

This theorem is a model for a general result in which the Euler obstruction of a space $X$ is related to the Euler obstruction of the pullback of the Jacobian module of the singularity whose pullback is $X$, with a defect term accounting for the difference.  If the map $F$ used to define $X$ has nice properties, then $X$ is a good approximation to the generic determinantal singularity, so it is expected that the Euler obstructions will be the same. In Corollary \ref{cor21}, we give conditions under this happens for determinantal singularities. The special feature for determinantal singularities is the small size of the fiber of the conormal variety for generic singularities. This implies that for a certain range of dimensions the polar varieties are unexpectedly empty. Exploring what the other  properties  generic determinantal varieties and their good approximations share would be interesting.

\section*{Acknowledgments}

The authors are grateful to Jonathan Mboyo Esole, Thiago de Melo, Otoniel Silva,  Jawad Snoussi and Xiping Zhang for their careful reading of the first draft of this paper and for their suggestions. 

The first author was partially supported by Conselho Nacional de Desenvolvimento Cient\'{i}fico e Tecnol\'{o}gico - CNPq, Brazil, grant PVE 401565/2014-9. The second author was partially supported by Funda\c{c}\~{a}o de Amparo \`{a} Pesquisa do Estado de S\~{a}o Paulo - FAPESP, Brazil,  grant 2015/16746-7 and Conselho Nacional de Desenvolvimento Cient\'{i}fico e Tecnol\'{o}gico - CNPq, Brazil, grants 474289/2013-3 and 303641/2013-4. The third author was partially supported by Funda\c{c}\~{a}o de Amparo \`{a} Pesquisa do Estado de S\~{a}o Paulo - FAPESP, Brazil, grant 2014/00304-2 and Conselho Nacional de Desenvolvimento Cient\'{i}fico e Tecnol\'{o}gico - CNPq, Brazil, grant 305651/2011-0. 

This paper was written while the second author was visiting Northeastern University, Boston, USA. During this period the first author had also visited the Universidade de São Paulo at São Carlos, Brazil, and we would like to thank these institutions for their hospitality.

%\section{Chern--Schwartz--MacPherson class of a singular variety}

\section{The Chern--Schwartz--MacPherson class of generic determinantal varieties}

In \cite{M}  MacPherson proved the existence and %unicity 
uniqueness of Chern classes for possibly singular complex algebraic varieties, which was conjectured earlier by Deligne and Grothendieck. These are homology classes which for nonsingular varieties are the Poincar\'{e} duals of the usual Chern classes. Some time later, Brasselet and Schwartz proved in \cite{BS}, using Alexander's duality that the Schwartz class, stated before the Deligne--Grothendieck conjecture, coincides with MacPherson's class, and therefore this class is called the Chern--Schwartz--MacPherson class.

In his proof, MacPherson used the language of constructible sets and functions. A constructible set in an algebraic variety is one obtained from the subvarieties by finitely many of the usual set-theoretic operations. A constructible function on a variety is one for which the variety has a finite partition into constructible sets such that the function is constant on each set. MacPherson proved the following result.

\begin{proposition}[{\cite[Prop.\ 1]{M}}] There is a unique covariant functor 
$\mathbf{F}$ from compact complex algebraic varieties to the abelian group whose value on a variety is the group of constructible functions from that variety to the integers and whose value $f_{*}$ on a map $f$ satisfies $$f_{*}(1_{W})(p)= \chi(f^{-1}(p)\cap W),$$where $1_{W}$ is the function that is identically one on the subvariety $W$ and zero elsewhere, and where $\chi$ denotes the topological Euler characteristic.
\end{proposition}

Theorem 1 of \cite{M} is the main result of that paper. As we mentioned before, the result  was conjectured by Deligne and Grothendieck and we write it below.

\begin{theorem} There exist a natural transformation from the functor $\mathbf{F}$ to homology which, on a nonsingular variety $X$, assigns to the constant function $1_{X}$ the Poincar\'{e} dual of the total Chern class of $X$.
\end{theorem}

In other words, the theorem asserts that we can assign to any constructible function $\alpha$ on a compact complex algebraic variety $X$ an element $c_{*}(\alpha)$ of $H_{*}(X)$ satisfying the following three conditions:
\begin{enumerate}

\item{$f_{*}c_{*}(\alpha)=c_{*}f_{*}(\alpha)$}
\item{$c_{*}(\alpha + \beta)=c_{*}(\alpha)+c_{*}(\beta)$}
\item{$c_{*}(1_{X})= \operatorname{Dual}  c(X)$ if $X$ smooth.}

\end{enumerate}

As mentioned in \cite{M}, this is exactly Deligne's definition of the total Chern class of any compact variety $X$; % it means that 
the total Chern class is 
$c_{*}$ applied to the constant function $1_{X}$ on $X$.
 The compactness restriction may be dropped with minor modifications of the proof if all maps are taken to be proper and Borel--Moore homology (homology with locally finite supports) is used.

Let us now introduce some objects in order to define the Chern--Schwartz--MacPherson class. For more details on these concepts we suggest \cite{BLS,BS,LT1,M}.

Suppose $X$ is a representative of a $d$-dimensional analytic germ $(X,0) \subset (\mathbb{C}^{n},0)$, such that $X \subset U$, where $U$ is an open subset of $\mathbb{C}^{n}$. Let $G(d,n)$ denote the Grassmannian of complex $d$-planes
in $\mathbb{C}^n$. On the regular part $X_{\text{reg}}$ of $X$ the Gauss map
$\phi: X_{\text{reg}} \to U\times G(d,n)$ is well defined by $\phi(x) =
(x,T_x(X_{\text{reg}}))$. 

\begin{definition}
The Nash transformation (or Nash blow up) of $X$ denoted by $N(X)$
is the closure of the image $\operatorname{Im}(\phi)$ in $ U\times G(d,n)$. It is a
(usually singular) complex analytic space endowed with an analytic
projection map $\nu : N(X) \to X$
which is biholomorphic away from $\nu^{-1}(\operatorname{Sing}(X))$.
\end{definition}

The fiber of  the tautological bundle  $\mathcal{T}$ over $G(d,n)$, at point $P \in G(d,n)$,
is the set of vectors $v$ in the $d$-plane $P$. We still denote by $\mathcal{T}$ the
corresponding trivial extension bundle over $ U \times G(d,n)$. Let $N(T)$ be the restriction of  ${\mathcal T}$\nota{} to $N(X)$, with
projection map $\pi$.  The bundle $N(T)$ on $N(X)$ is called \textit{the  Nash bundle} of $X$.

An element of $N(T)$ is written $(x,P,v)$ where $x\in U$,
$P$ is a $d$-plane in ${\mathbb C}^n$ based at $x$ and $v$ is a vector in
$P$. We have the following diagram:\nota{se quiser, posso melhorar diag}
$$
\begin{matrix}
N(T) & \hookrightarrow & {\mathcal T} \cr
{\pi} \downarrow & & \downarrow \cr
N(X) & \hookrightarrow & U \times G(d,n) \cr
{\nu}\downarrow & & \downarrow \cr
X & \hookrightarrow & U. \cr
\end{matrix}
$$

Mather has defined an extension of Chern classes to singular varieties by the formula$$c_{CM}(X)= \nu_{*}\operatorname{Dual}  c(N(T)),$$where $c(N(T))$ denotes the total Chern class in cohomology of the Nash bundle, the Dual denotes the Poincar\'{e} duality map defined by capping with the fundamental (orientation) homology class.

An algebraic cycle on a variety $X$ is a finite formal linear sum $\sum n_{i}[X_{i}]$ where the $n_{i}$ are integers and the $X_{i}$ are irreducible subvarieties of $X$. We may define $c_{CM}$ on any algebraic cycle of $X$ by$$c_{CM}(\sum n_{i} [X_{i}])=\sum n_{i}c_{CM}(X_{i}),$$where by abuse of notation we denote ${\operatorname{incl}_{i}}_{*}c_{CM}(X_{i})$ by $c_{CM}(X_{i})$.

An important object introduced by MacPherson in his work is the local Euler obstruction. This invariant was deeply investigated by many authors, and for an overview about it see \cite{B}.  Brasselet and Schwartz presented in \cite{BS} an alternative definition for the local Euler obstruction using stratified vector fields. 

\begin{definition}
Let us denote by $TU|_{X}$ the restriction to $X$ of the tangent bundle of $U$. A stratified vector field $v$ on $X$ means a continuous section of $TU|_{X}$ such that if $x \in V_{\alpha} \cap X$ then $v(x) \in T_{x}(V_{\alpha})$. 
\end{definition}

By Whitney condition (a) one has the following:

\begin{lemma}[See \cite{BS}]
Every stratified vector field $v$ nowhere zero on a subset $A \subset X$ has a canonical lifting as a nowhere zero section $\tilde{v}$ of the Nash bundle $N(T)$ over $\nu^{-1}(A) \subset N(X)$.
\end{lemma}

Now consider a stratified radial vector field $v(x)$ in a neighborhood of $\{0\}$ in $X$, {\textit{i.e.}}, there is $\varepsilon_{0}$ such that for every $0<\varepsilon \leq \varepsilon_{0}$, $v(x)$ is pointing outwards the ball ${B}_{\varepsilon}$ over the boundary ${S}_{\varepsilon} := \partial{{B}_{\varepsilon}}$.

The following interpretation of the local Euler obstruction has been given by Brasselet and Schwartz in \cite{BS}.

\begin{definition}
Let $v$ be a radial vector field on $X \cap {S}_{\varepsilon}$ and $\tilde{v}$ the lifting of $v$ on $\nu^{-1}(X \cap {S}_{\varepsilon})$ to a section of the Nash bundle. 

The local Euler obstruction (or simply the Euler obstruction), denoted by $\operatorname{Eu}_{0}(X)$, is defined to be the obstruction to extending $\tilde{v}$ as a nowhere zero section of $N(T)$ over $\nu^{-1}(X \cap {B}_{\varepsilon})$.
\end{definition}

More precisely, let $$\mathcal{O}{(\tilde{v})} \in H^{2d}(\nu^{-1}(X \cap {B_{\varepsilon}}), \nu^{-1}(X \cap {S_{\varepsilon}});\mathbb{Z})$$ be the obstruction cocycle to extending $\tilde{v}$ as a nowhere zero section of $\widetilde{T}$\nota{wide?} inside $\nu^{-1}(X\cap {B_{\varepsilon}})$. The Euler obstruction $\operatorname{Eu}_{0}(X)$ is defined as the evaluation of the cocycle $\mathcal{O}(\tilde{v})$ on the fundamental class of the topological  pair $(\nu^{-1}(X \cap {B_{\varepsilon}}), \nu^{-1}(X \cap {S_{\varepsilon}}))$. %The Euler obstruction is an integer.

Note that if $0 \in X$ is a smooth point we have $\operatorname{Eu}_{0}(X)=1$, but the converse is false, this was first observed by Piene in \cite{Piene2} (Example, pp.\ 28--29).
 
In this paper we use an interesting formula for the local Euler obstruction due to Brasselet, L\^{e} and Seade, that shows that
the Euler obstruction, as a constructible function, satisfies the Euler condition
relative to generic linear forms.

\begin{theorem}[{\cite[Theo.\ 3.1]{BLS}}] \label{BLS}Let $(X,0) \subset (\mathbb{C}^n,0)$ be an equidimensional reduced complex analytic germ of dimension $d$. Let us consider $X \subset U \subset \mathbb{C}^n$ a sufficiently small representative of the germ, where $U$ is an open subset of $\mathbb{C}^{n}$. We consider a complex analytic Whitney stratification $\mathcal{V} = \{V_i\}$ of $U$ adapted to $X$ and we assume that $\{0\}$ is a $0$-dimensional stratum. We also assume  that $0$ belongs to the closure of all the strata. Then
for each generic linear form $l$, there is $\varepsilon_0$
such that for any $\varepsilon$ with $0<\varepsilon<
\varepsilon_0$ and $\delta \neq 0$ sufficiently small, the Euler obstruction of $(X,0)$ is equal to:
$$\operatorname{Eu}_{0}(X)=\sum_{i =1}^q\chi \big(V_i\cap B_\varepsilon\cap l^{-1}(\delta) \big) \cdot
\operatorname{Eu}_{V_i}(X),$$ where $\operatorname{Eu}_{V_i}(X)$ is the value of the Euler
obstruction of $X$ at any point of $V_i$, $i=1,\dots,q$, and $0 < \vert \delta \vert \ll \varepsilon \ll 1$, where $\chi$ denotes the topological Euler characteristic.
\end{theorem}

With the aid of Gonzalez--Sprinberg's purely algebraic interpretation of the local Euler obstruction (\cite{GS2}), L\^e and Teissier in  \cite{LT1}  showed that the local Euler obstruction  is an alternating sum of the multiplicities  of the  local polar varieties. 
This is an important formula for computing the local Euler obstruction, and we use it in this paper. 

\begin{theorem}[{\cite[Cor.\ 5.1.4]{LT1}}]\label{LT} Let $(X,0)\subset (\mathbb{C}^{n+1},0)$ be the germ of an equidimensional reduced analytic  space  of dimension $d$. Then
  $$\operatorname{Eu}_0(X)=\sum_{i=0}^{d-1}(-1)^{d-i-1}m_{d-i-1}(X,0),\nota{sum}$$
where  $m_i(X,0)$ is the polar multiplicity
  of the  local polar varieties $P_i(X,0)$.
  \end{theorem}

The local polar variety $P_{i}(X,0)$ is the local version of the global polar variety $P_{i}(X)$  (Definition \ref{polar}).

Using the local Euler obstruction, MacPherson defined a map $T$ from the algebraic cycles on $X$ to the constructible functions on $X$ by$$T(\sum n_{i}X_{i})(p)= \sum n_{i} \operatorname{Eu}_{p}(X_i).$$ And he proved that (Lemma 2 and Theorem 2 of \cite{M}): 

\begin{theorem}$T$ is a well-defined isomorphism from the group of algebraic cycles to the group of constructible functions and that $c_{CM}T^{-1}(1_{X})$ satisfies the requirements for $c_{*}$ in the Deligne--Grothendieck conjecture.
\end{theorem}

%\section{The Chern--MacPherson--Schwartz Chern class of generic determinantal varieties}

In this section we prove a formula to compute the local Euler obstruction of a generic determinantal variety, and applying  MacPherson's definition, we find the Chern--Schwartz--MacPherson class of this variety.

First, let us recall the definition of the generic determinantal variety.

\begin{definition}\label{DefGenDet}Let  $n, k, s \in \mathbb{Z}$, $n\geq 1$, $k \geq 0$ and $\operatorname{Mat}_{(n,n+k)}(\mathbb{C})$ be the set of all $n\times (n+k)$ matrices with
complex entries, $\Sigma^{s}\subset \operatorname{Mat}_{(n,n+k)}(\mathbb{C})$ the subset
formed by matrices that have rank less than $s$, with $1\leq s\leq
n$. The set $\Sigma^{s}$ is called the \textit{generic
determinantal variety}.
\end{definition}

\begin{remark} The following  properties of the generic determinantal varieties  are fundamental in this work. 
\begin{enumerate}

\item{$\Sigma^{s}$ is an irreducible
singular algebraic variety.}

\item{The codimension of $\Sigma^{s}$ in the ambient space is $(n-s+1)(n+k-s+1)$.} 

\item{The singular set of $\Sigma^{s}$ is exactly
$\Sigma^{s-1}$.}

\item{ The stratification of $\Sigma^{s}$ given by  $\{\Sigma^{t}\setminus \Sigma^{t-1}\}$, with $1 \leq t \leq s$ is locally analytically trivial and hence it is a Whitney stratification of $\Sigma^{s}$.}

\end{enumerate}

As references for these topics we recommend, chapter 2, section 2, of \cite{ACGH} and the book \cite{Bruns}.

\end{remark}

\begin{remark}

Every element of $\operatorname{Mat}_{(n, n+k)}(\mathbb{C})$ can be seen as a linear
map  from $\mathbb{C}^n$ to $\mathbb{C}^{n+k}$, or from $\mathbb{C}^{n+k}$ to $\mathbb{C}^n$, then we will also refer to the space of matrices as
$\operatorname{Hom}(\mathbb{C}^n, \mathbb{C}^{n+k})$ or $\operatorname{Hom} (\mathbb{C}^{n+k}, \mathbb{C}^n)$.
\end{remark}

The next result is a very important in this paper. To state it let us fix some notations. Let $\overline{\chi}$ denote the reduced Euler characteristic, that is $\overline{\chi}= \chi -1$, where $\chi$ denotes the topological Euler characteristic.

Let us also recall the notion of normal slice this notion is related to the complex link  and normal Morse datum. The complex link is an important object in the study of the topology of complex analytic sets. It is analogous to the Milnor fibre and was studied first in \cite{Le1}. It plays a crucial role in complex stratified Morse theory (see \cite{GM}) and appears in general bouquet theorems for the Milnor fibre of a function with isolated singularity (see \cite{Le2, Si, Ti}). It is related to the multiplicity of polar varieties and also to the local Euler obstruction (see \cite{LT1, LT2}).

\begin{definition}
Let $V$ be a stratum of a Whitney stratification of $X$, a small representative of the analytic germ $(X,0) \subset (\mathbb{C}^{n},0)$, and $x$ be a point in $V$. We call $N$ a normal slice to $V$ at $x$, if $N$ is a closed complex submanifold of $\mathbb{C}^n$ which is transversal to $V$ at $x$ and $N \cap V =\{x\}$.
\end{definition}
%\begin{definition}
%The complex link $l_V$ of $V$ is defined by:
%$$l_V = X\cap N \cap B_{\epsilon}(x)\cap \{g=\delta\} ,$$
%where $ 0< \vert \delta \vert \ll \epsilon \ll 1$.  Here $B_{\epsilon}(x)$ is the closed ball of radius $\epsilon$ centered at $x$ .

%The normal Morse datum ${\rm NMD}(V)$ of $V$ is the pair of spaces: 
%$${\rm NMD}(S) =(X\cap N \cap B_{\epsilon}(x), X\cap N \cap B_{\epsilon}(x)\cap \{g=\delta\}).$$
%\end{definition}
%The fact that these two notions are well-defined, i.e. independent of all the choices made to define them, is explained in \cite{GM}. 

\begin{proposition}[{\cite[Prop.\ 3]{EG}}]\label{Prop3} Let $\ell : \operatorname{Hom}(\mathbb{C}^n, \mathbb{C}^{n+k})\to \mathbb{C}$ be a generic linear form. Then, for $s  \leq n$, one has
$$\overline{\chi} (\Sigma^{s} \cap {\ell}^{-1}(1))=(-1)^{s} \binom{n-1}{s-1}.\nota{usei binom bem mais facil}$$
%Where $\overline{\chi}$ denotes reduced Euler characteristic.
\end{proposition}

In order to find  the Chern--Schwartz--MacPherson  class of a generic determinantal variety, first we calculate its local Euler obstruction to apply MacPherson's result. To simplify notation, we denote from here the local Euler obstruction of the generic determinantal variety  $\Sigma^{s} \subset \operatorname{Hom}(\mathbb{C}^n, \mathbb{C}^{n+k})$ at a point $p$ by $e_{p}(s,n)$, and if $p=0$ we denote it only by $e(s,n)$. We do not use $k$ in the notation because, as we see in the next result, the formula does not depend on $k$.

\begin{lemma}\label{lemma1} The Euler obstruction of $\Sigma^{s} \subset \operatorname{Hom}(\mathbb{C}^n, \mathbb{C}^{n+k})$ at the origin satisfies the recurrence relation:
$$e(s,n)= \sum_{i=2}^{s}(\overline{\chi}(i,n)-\overline{\chi}(i-1,n))e(s-i+1,n-i+1).$$
\end{lemma}

\begin{proof} 
Applying Theorem \ref{BLS} in our case, as $\Sigma^{1} =\{0\}$ and $\Sigma^{0}= \emptyset$, taking a generic linear form $l$ we have 
$$e(s,n)= \sum_{i=2}^{s} \chi ((\Sigma^{i}\setminus \Sigma^{i-1}) \cap B_{\varepsilon} \cap l^{-1}(t_0))e_{p_i}(s,n),$$where $p_i \in \Sigma^{i}\setminus \Sigma^{i-1}$, $0<  |t_0|\ll 1$.

As we work with the topological Euler characteristic, by the inclusion-exclusion principle we have  $$\chi((\Sigma^{i}\setminus \Sigma^{i-1}) \cap B_{\varepsilon} \cap l^{-1}(t_0))=\chi(\Sigma^{i}\cap B_{\varepsilon} \cap l^{-1}(t_0)) -\chi((\Sigma^{i-1}\cap B_{\varepsilon} \cap l^{-1}(t_0)).$$ 

To simplify notation we rewrite  $$\chi(\Sigma^{i} \cap B_{\varepsilon} \cap l^{-1}(t_0)) = \chi(i,n).$$

Since the stratification is locally holomorphically trivial, working in a neighborhood of $\Sigma^{s},p_{i}$,  we have that $\Sigma^{s}$ is analytic equivalent to $(\Sigma^{s} \cap N )\times B$, where $B$ is a ball of complementary dimension inside $\Sigma^{i}\setminus \Sigma^{i-1}$ and $N$ is a normal slice at $p_i$. Therefore, we can compute $e_{p_{i}}(s,n)$ as follows:$$e_{p_{i}}(s,n)= \operatorname{Eu}_{p_{i}}(\Sigma^{s} \cap N) \cdot \operatorname{Eu}_{p_i}(B) = \operatorname{Eu}_{p_{i}}(\Sigma^{s} \cap N),$$(see  \cite[pp.\ 423]{M}).

As in \cite{EG}, we can use the action of $\operatorname{GL}_n(\mathbb C)\times \operatorname{GL}_{n+k}(\mathbb C)$ into $\operatorname{Hom}(\mathbb{C}^n, \mathbb{C}^{n+k})$, to assume 
%Following \cite{EG} ideas, due to the group action we can assume 
$p_i$ is a block matrix with an $i-1$\nota{pq ()?} identity matrix  in the upper left corner and the other blocks $0$. Then the normal slice is just the lower right block. If we take a matrix in this lower right block whose rank is $k$, then the rank of the whole matrix is $k+i-1$, so the matrices of rank $s-i$ in the lower right block are the matrices of rank $s-1$ in the normal slice. Therefore  $\{\Sigma^{s} \cap N\}$ is isomorphic to $\Sigma^{s-(i-1)} \subset \operatorname{Hom}(n-i+1,n-i+k+1)$ and therefore we have that,
\begin{equation}\label{cut}e_{p_i}(s,n)=e(s-i+1, n-i+1).\end{equation} 

Using this information we can rewrite the formula as follows.
$$e(s,n)= \sum_{i=2}^{s}(\chi(i,n)-\chi(i-1,n))e(s-i+1,n-i+1).$$

Or, using the reduced Euler characteristic, we can rewrite the formula as:
%Using the reduced Euler characteristic, denoted by $\overline{\chi}$ and defined as $\overline{\chi}= \chi -1$, where $\chi$ denotes the topological Euler characteristic, we can rewrite this formula as:
$$e(s,n)= \sum_{i=2}^{s}(\overline{\chi}(i,n)-\overline{\chi}(i-1,n))e(s-i+1,n-i+1).$$
\end{proof}

The next result is a technical result about alternating sums of products of binomial numbers that we need in the sequel. For techniques in combinatorics we refer \cite{CK}.

\begin{lemma}\label{lemma2}
$$\sum_{i=2}^{s} (-1)^{i}\begin{pmatrix} n-1\\ i-1\end{pmatrix} \begin{pmatrix} n-i\\ s-i\end{pmatrix} = \begin{pmatrix} n-1\\ s-1\end{pmatrix} .\nota{binom?}$$
\end{lemma}

\begin{proof}
There are 2 cases depending on the parity of $s$.

%\begin{description} \item[Case 1] 
\medskip\noindent\textbf{Case 1.} Let us first prove the $s$ even case.
Comparing $\begin{pmatrix} n-1\\ i-1\end{pmatrix} \begin{pmatrix} n-i\\ s-i\end{pmatrix}$ and $\begin{pmatrix} n-1\\ s-i\end{pmatrix} \begin{pmatrix} n-(s+1-i)\\ s-(s+1-i)\end{pmatrix}$, they are the same. The signs in the sum are $(-1)^{i}$ and $(-1)^{s+1-i}$ which differ for $s$ even, so the terms cancel, except for the last term $$(-1)^{s} \begin{pmatrix} n-1\\ s-1\end{pmatrix} = \begin{pmatrix} n-1\\ s-1\end{pmatrix}.$$

%\item[Case 2]
\medskip\noindent\textbf{Case 2.} Let us now prove the $s$ odd case.
\begin{align*}
\begin{pmatrix} n-1\\ i-1\end{pmatrix} \begin{pmatrix} n-i\\ s-i\end{pmatrix} &= \frac{(n-1)\cdots (n-i+1)}{(i-1)!}  \frac{(n-i)\cdots (n-s+1)}{(s-i)!}\\ &= \frac{(n-1)!}{(n-s)!(i-1)! (s-i)!}
\end{align*} and therefore  the sum can be rewritten as
\begin{multline*}
\frac{(n-1)!}{(n-s)!(s-1)!}(-1)\sum_{i=2}^{s}(-1)^{i-1}\frac{(s-1)!}{(i-1)!(s-i)!} = \\(-1)\frac{(n-1)!}{(n-s)!(s-1)!} \sum_{i=1}^{s-1}(-1)^{i}\frac{(s-1)!}{(i)!(s-1-i)!}
\end{multline*}
%$$ \frac{(n-1)!}{(n-s)!(s-1)!}(-1)\sum_{i=2}^{s}(-1)^{i-1}\frac{(s-1)!}{(i-1)!(s-i)!}=$$
%
%$$ (-1)\frac{(n-1)!}{(n-s)!(s-1)!} \sum_{i=1}^{s-1}(-1)^{i}\frac{(s-1)!}{(i)!(s-1-i)!}$$
Now, 
$$0= ((-1)+1)^{s-1}= \sum_{i=0}^{s-1}(-1)^{i}\begin{pmatrix} s-1\\ i\end{pmatrix}$$
and so $$\sum_{i=1}^{s-1}(-1)^{i}\frac{(s-1)!}{i!(s-1-i)!}= -1.$$
But $$ \frac{(n-1)!}{(s-1)!((n-1)-(s-1))!} =  \begin{pmatrix} n-1\\ s-1\end{pmatrix}.$$
Therefore  $$\sum_{i=2}^{s} (-1)^{i}\begin{pmatrix} n-1\\ i-1\end{pmatrix} \begin{pmatrix} n-i\\ s-i\end{pmatrix}= \begin{pmatrix} n\\ s-1\end{pmatrix}.$$
%Applaying these in our previous equation  we have$$e(s,n) = \sum_{i=0}^{s-1} \begin{pmatrix}n-s+i\\ i\end{pmatrix},$$for $1 \leq s \leq n$.
 %But, if we apply the Chuh Shih Chieh identity (\cite{CK} p79 example 2.5.1 ii)
%$$\binom{k} {0}+\binom{k+1} {1}\dots +\binom{k+l} {l}=\binom{k+l+1} {l}$$
%to this formula  with $k=n-s$, and we get 
%$$e(s,n) =\binom{n} {s-1}.$$
%\end{description}
\end{proof}

\begin{theorem}\label{EulerFormula} Let $\Sigma^{s} \subset \operatorname{Hom}(\mathbb{C}^n, \mathbb{C}^{n+k})$ be a generic determinantal variety defined as above, we have $$e(s,n) =\binom{n} {s-1},$$for $1 \leq s \leq n$.
\end{theorem}

\begin{proof} To prove this result we use Lemma \ref{lemma1}, Lemma \ref{lemma2} and induction.

First note that $e(1,n)=1$, for $n \in \mathbb{Z}$,  $n \geq 1$. In this case $\Sigma^{1} =\{0\}$, and  the Local Euler obstruction of one point is its Euler characteristic.

From Lemma \ref{lemma1} we have
$$e(s,n)= \sum_{i=2}^{s}(\overline{\chi}(i,n)-\overline{\chi}(i-1,n))e(s-i+1,n-i+1).$$

Distributing the Euler obstruction terms and expanding  the sum we have:
\begin{align*}
\overline{\chi}(2,n)e(s-1,n-1) &- \overline{\chi}(1,n)e(s-1,n-1) +\\
\overline{\chi}(3,n)e(s-2,n-2) &- \overline{\chi}(2,n)e(s-2,n-2) +\\
\overline{\chi}(4,n)e(s-3,n-3) &- \overline{\chi}(3,n)e(s-3,n-3) +\\
 & \vdots \\
\overline{\chi}(s,n)e(1,n-s+1) &- \overline{\chi}(s-1,n)e(1,n-s+1). 
\end{align*}

And as $\overline{\chi}(1,n) = -1$, we can rewrite the formula as:
$$e(s,n)= 1 \cdot e(s-1,n-1) + \sum_{i=2}^{s} \overline{\chi}(i,n)(e(s-i+1,n-i+1)-e(s-i,n-i)).$$

Let us consider $e(s-i+1,n-i+1)- e(s-i,n-i)$. Using the inductive hypothesis we have that
\begin{multline*}
e(s-i+1,n-i+1)- e(s-i,n-i) = \\
\binom{n-i+1}{s-1} - \binom{n-i}{s-i-1}
= \binom{n-i}{s-i}.
\end{multline*}

Therefore it follows that 
\begin{equation}
e(s,n)=e(s-1,n-1)+ \sum_{i=2}^{s}  \overline{\chi}(i,n) \binom{n-i}{s-i}.
\end{equation}

By Proposition \ref{Prop3} we have $$\overline{\chi}(i,n) = (-1)^{i} \binom{n-1}{i-1},$$ so we get 
$$e(s,n)= e(s-1,n-1) + \sum_{i=2}^{s} (-1)^{i}\binom{n-1}{i-1} \binom{n-i}{s-i}.$$

Now, from Lemma \ref{lemma2} we have 
$$\sum_{i=2}^{s} (-1)^{i}\binom{n-1}{i-1} \binom{n-i}{s-i}  = \binom{n-1}{s-1}.$$

Then $e(s,n)=e(s-1,n-1)+\binom{n-1}{s-1}.$ By the inductive hypothesis we have
 $e(s,n)=e(s-1,n-1)+ \binom{n-1}{s-1} = \binom{n-1}{s-2}  + \binom{n-1}{s-1} = \binom{n}{s-1}.$
%Applaying these in our previous equation  we have$$e(s,n) = \sum_{i=0}^{s-1} \begin{pmatrix}n-s+i\\ i\end{pmatrix},$$for $1 \leq s \leq n$.
 %But, if we apply the Chuh Shih Chieh identity (\cite{CK} p79 example 2.5.1 ii)
%$$\binom{k} {0}+\binom{k+1} {1}\dots +\binom{k+l} {l}=\binom{k+l+1} {l}$$
%to this formula  with $k=n-s$, and we get 
%$$e(s,n) =\binom{n} {s-1}.$$
\end{proof}

\begin{remark}

 The last result has a pretty accompanying graphic. For this part, fix $k \in \mathbb{Z}^{+}$, $k \geq 1$, and for  $i \in \mathbb{Z}^{+}$, $1\leq i \leq n+1$, let us denote $\Sigma^{i} \subset \operatorname{Hom}(n,n+k)$ by $\Sigma^{i}_{n}$.

On the one hand we have a triangle of spaces and maps. In the apex of the triangle (row zero) we have $0\in \operatorname{Hom}(\mathbb{C}^0, \mathbb{C}^k)$. Row $1$ is  $0\in \operatorname{Hom}(\mathbb{C}^1, \mathbb{C}^{k+1})$ and  $\operatorname{Hom}(\mathbb{C}^1, \mathbb{C}^{k+1})$. We have maps from the element in row $0$ to each element in row $1$ given by the inclusions of $\mathbb{C}^k$ to $\mathbb{C}^{k+1}$, and projection of $\mathbb{C}^1$ to $\mathbb{C}^0$. Row $2$ is $\Sigma^1 = \{0\}\in \operatorname{Hom}(\mathbb{C}^2, \mathbb{C}^{k+2})$, $\Sigma^2\subset \operatorname{Hom}(\mathbb{C}^2, \mathbb{C}^{k+2})$, and $\operatorname{Hom}(\mathbb{C}^2, \mathbb{C}^{k+2})$.
  
Again there are maps given by projection and inclusion from   elements of row $1$ into adjacent pairs of elements of row $2$. Then row $n$ consists of the spaces $\Sigma^{i}_{n}\subset \operatorname{Hom}(\mathbb{C}^n, \mathbb{C}^{n+k})$, $1\le i\le n+1$, with maps from the previous row to adjacent pairs of elements of this row. 
The other triangle is Pascal's triangle. Then our theorem says that the local Euler obstruction takes the triangle of spaces to Pascal's triangle. 
\begin{figure}[!htb]
\centering
\includegraphics[scale=0.8] {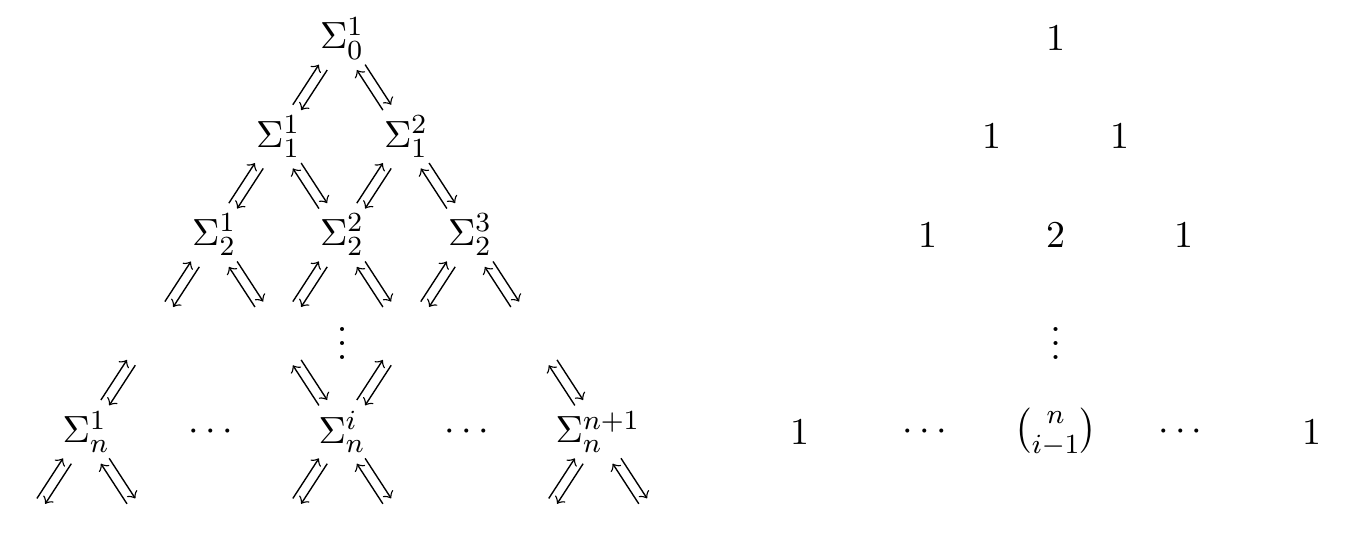}
\caption{Triangles}
%\label{Rotulo}
\end{figure}
\end{remark}

%\begin{remark}
%The Euler obstruction is a constructible function,  $$e_{p}(s,n)=\sum_{i=1}^s %\alpha_{i}{\mathds 1}_{\Sigma^i}(p),$$ where ${\mathds 1} _{\Sigma^i}$ is $1$ on $\Sigma^i$, $0$ elsewhere and $\alpha_{i}$ is the value of the Euler obstruction of $\Sigma^{n}$ at a generic point of $\Sigma^i$. As we saw in the proof of Lemma \ref{lemma1}, $\alpha_{i}=e(s-i+1, n-i+1)$. Using our formula to calculate $\alpha_{i}$ we have the next  corollary.
%\end{remark}

%\begin{corollary}In the above setup we have 
%$$ e_{p}(s,n)=\sum_{i=1}^s \binom{n-i+1}{s-i}{\mathds 1} _{\Sigma^i}(p).$$
%\end{corollary}

\begin{remark}
The Euler obstruction is a constructible function,  $$e_{p}(s,n)=\sum_{i=1}^s \alpha_{i}{\mathds 1}_ {(\Sigma^i\setminus \Sigma^{i-1})}(p),$$ where ${\mathds 1}_ {(\Sigma^i\setminus \Sigma^{i-1})}(p)$ is $1$ on $\Sigma^i\setminus \Sigma^{i-1}$, $0$ elsewhere and $\alpha_{i}$ is the value of the Euler obstruction of $\Sigma^{n}$ at any  point of $\Sigma^i \setminus \Sigma^{i-1}$. As we saw in the proof of Lemma \ref{lemma1}, $\alpha_{i}=e(s-i+1, n-i+1)$. Using our formula to calculate $\alpha_{i}$ we have the next  corollary.
\end{remark}

\begin{corollary}In the above setup we have 
$$ e_{p}(s,n)=\sum_{i=1}^s \binom{n-i+1}{s-i}{\mathds 1} _{(\Sigma^i\setminus \Sigma^{i-1})}(p).$$
\end{corollary}

Recall that an algebraic cycle on a variety $X$ is a finite formal linear sum $\sum n_{i}[X_{i}]$ where the $n_{i}$ are integers and the $X_{i}$ are irreducible subvarieties of $X$. Taking $X= \Sigma^{s}$ and remembering that all $\Sigma^{i} \subset \Sigma^{s}$, where $i= 1, \dots, s$ are irreducible subvarieties and using Theorem \ref{EulerFormula}  we get a formula for the local Chern--Schwartz--MacPherson cycle of $\Sigma^{n}$, denoted by $[\operatorname{csm}(\Sigma^{s})]$.

The next result is an interesting property of alternating sums of binomials products, whose elements are chosen in the Pascal's triangle in a ``V'' distribution. This lemma is essential to define the Chern--Schwartz--MacPherson cycle.

\begin{lemma}\label{V-property}

$$ \sum_{i=0}^{s-1} (-1)^{s-1+i}\binom{n-i-1}{s-i-1}\binom{n}{i}=1.$$

\end{lemma}

\begin{proof}

The  proof is  by induction, using Lemma \ref{lemma2}.

For any $n \in \mathbb{N}$ and $s=1$ we have$$(-1)^{0}\binom{n-1}{0}\binom{n}{0}=1.$$
We suppose by induction that the formula is true for $s-1$, that is:$$\sum_{i=0}^{s-2}(-1)^{s-2+i}\binom{n-i-1}{s-i-2}\binom{n}{i}=1.$$
We can rewrite \begin{equation}\label{Binomial}\sum_{i=0}^{s-1} (-1)^{s-1+i}\binom{n-i-1}{s-i-1}\binom{n}{i}\end{equation}as
$$(-1)^{s-1}\binom{n-1}{s-1} + \sum_{i=1}^{s-1} (-1)^{s-1+i}\binom{n-i-1}{s-i-1}\binom{n}{i}.$$
Making the change of variables $j=i+1$, we get 
%$$\text{Equation } \ref{Binomial} = (-1)^{s-1}\binom{n-1}{s-1}+ \sum\limits_{j=2}^{s}(-1)^{s+j}\binom{n-j}{s-j}\binom{n}{j-1}.$$
$$\text{Equation \eqref{Binomial}} = (-1)^{s-1}\binom{n-1}{s-1}+ \sum\limits_{j=2}^{s}(-1)^{s+j}\binom{n-j}{s-j}\binom{n}{j-1}.$$
Using Lemma \ref{lemma2} we have,
\begin{multline*}
\text{Equation \eqref{Binomial}} = (-1)^{s-1}\binom{n-1}{s-1}+\\ (-1)^{s} \sum\limits_{j=2}^{s}(-1)^{j}\binom{n-j}{s-j}\left[\binom{n-1}{j-1}+\binom{n-1}{j-2}\right].
\end{multline*}
Then
\begin{multline*}
\text{Equation \eqref{Binomial}} = (-1)^{s-1}\binom{n-1}{s-1}+ (-1)^{s}\binom{n-1}{s-1}+\\ (-1)^{s} \sum\limits_{j=2}^{s}(-1)^{j}\binom{n-j}{s-j}\binom{n-1}{j-2}.
\end{multline*}
Now,\enlargethispage{\baselineskip}
$$(-1)^{s} \sum\limits_{j=2}^{s}(-1)^{j}\binom{n-j}{s-j}\binom{n-1}{j-2}= \sum\limits_{i=0}^{s-2}(-1)^{s-2+i}\binom{n-i-1}{s-i-2}\binom{n}{i}=1. $$
\end{proof}

\begin{theorem} The local Chern--Schwartz--MacPherson cycle of the algebraic variety   $\Sigma^{s}\subset \operatorname{Hom}(n,n+k)$ is $$[\operatorname{csm}(\Sigma^{s})]= \sum_{i=0}^{s-1} (-1)^{s-1+i} \binom{n-i-1}{s-i-1}{[\Sigma^{i+1}]}$$\end{theorem}
\begin{proof}

As remarked in $(6.1.5)$ of \cite{LT1}, the essential difficulty to calculate the Chern--Schwartz--MacPherson class of a singular analytic space $X$ is to construct a cycle $\sum n_{j}[X_{j}]$ of $X$ such that
$$\sum n_{j}\operatorname{Eu}_{p}(X_{j})=1,$$for every $p\in X$. A cycle that satisfies this condition is called Chern--Schwartz--MacPherson cycle of the variety. For more details for the construction in a more general setting see Section 6 of \cite{LT1}.

It means that, in our case, we only need to show that $$\operatorname{Eu}_p(\sum_{i=0}^{s-1} (-1)^{s-1+i} \binom{n-i-1}{s-i-1}{[\Sigma^{i+1}]})=1$$ for all points $p\in \Sigma^{s}$.
$$\operatorname{Eu}_p(\sum_{i=0}^{s-1} (-1)^{s-1+i} \binom{n-i-1}{s-i-1}{[\Sigma^{i+1}]})$$
becomes 
$$\sum_{i=0}^{s-1} (-1)^{s-1+i}\binom{n-i-1}{s-i-1}{e_{p}(i+1,n)}.$$

Note first that if $p=0$, using the Theorem \ref{EulerFormula}, this last sum is exactly the sum at Lemma \ref{V-property}, so 
$$\sum\limits_{i=0}^{s-1} (-1)^{s-1+i} \binom{n-i-1}{s-i-1}{e_{p}(i+1,n)} = \sum_{i=0}^{s-1} (-1)^{s-1+i}\binom{n-i-1}{s-i-1}\binom{n}{i}=1.$$

Let us take $p \in \Sigma^j \setminus \Sigma^{j-1}$, $1 \leq j \leq s$. As $p \in \Sigma^{i}$ only for $i \geq j$ and using the relation \eqref{cut}  it is:
$$\sum_{i=j-1}^{s-1} (-1)^{s-1+i} \binom{n-i-1}{s-i-1}e(i-j+2,n-j+1).$$

From Theorem \ref{EulerFormula}  we get
$$\sum_{i=j-1}^{s-1}(-1)^{s-1+i} \binom{n-i-1}{s-i-1}\binom{n-j+1}{i-j+1}.$$

Now, if we make a change of variables with $N= n-j+1$ and $S= s-j+1$ and $l= i-j+1$, we get
$$\sum_{l=0}^{S-1} (-1)^{S-1+l} \binom{N-l-1}{S-l-1}\binom{N}{l},$$and from Lemma \ref{V-property} we have that this sum  is $1$, which finishes the proof. 
\end{proof}

Using this last proposition and by MacPherson definition, we can calculate  the Chern--Schwartz--MacPherson class of $\Sigma^{s}$ as follows.

\begin{theorem}\label{Theo}
In the same setting as above, we have that the total Chern--Schwartz--MacPherson  class of $\Sigma^{s}$ 
is $$c_{\textup{CSM}}(\Sigma^{s})=\sum_{j=0}^{s-1} (-1)^{s-1+j} \binom{n-j-1}{s-j-1}c_{\textup{CM}}(\Sigma^{j+1})$$where $c_{\textup{CM}}(X)$ denotes the total Chern--Mather class of a variety $X$.
\end{theorem}

In turn by \cite{ST}, we are able to represent these total Chern classes using representatives of the polar varieties of the $\Sigma^{j}$. In order to express the Chern--Mather classes in terms of polar varieties in \cite{ST}, Tibar and Schurmann used a ``general transversality'' result of Kleiman \cite{Kleiman}.  It was also used in the local analytic context by Teissier in \cite{Te} to establish the existence and the main properties of his ``generic local polar varieties''. Piene proved the corresponding global result for projective varieties \cite{Piene1,Piene2}. 

\begin{definition}\label{polar} Let $X \subset \mathbb{C}^{n}$ be an equidimensional algebraic variety of dimension $d < n$. The $k$-th global polar variety of $X $ $(0\leq k \leq d)$ is the following algebraic set:
$$P_{k}(X)= \overline{\operatorname{Crit}(x_{1}, \dots, x_{d-k+1})_{\mid X_{\textup{reg}}}},$$
with $\operatorname{Crit}(x_{1},\dots,x_{d-k+1})_{\mid X_{\textup{reg}}}$ the usual critical locus of points $x \in X_{\textup{reg}}$ where the differentials of these functions restricted to $X_{\text{reg}}$ are linearly dependent. For general coordinates $x_{i}$, the polar variety $P_{k}(X)$ has codimension $k$ or is empty, for all $0<k\leq d$.  We also can denote the polar varieties by dimension as $P^{d-k}(X)$. This notation will be also used in this text when convenient.
\end{definition}

\begin{proposition}[{\cite[Prop.\ 2.7]{ST}, \cite[Theo.\ 3]{Piene2}}] Let $X \subset \mathbb{C}^{n}$ be an equidimensional algebraic variety of dimension $d < n$. Then, for each $k$, for the $k$-th global polar variety of $X $ we have,
$$c_{\textup{CM}}^{d-k}(X)= (-1)^{d-k}[P_{k}(X)] \in H_{d-k}(X).$$
\end{proposition}

Using the last result and the Theorem \ref{Theo} we get the following relation.

\begin{proposition} Let us denote $d= \dim \Sigma^{s}$ and $d_{i} = \dim \Sigma^{i}$, $0\leq i < s$. In the same setup as above, we have,
$$[P_{d-d_{i+1}}(\Sigma^{s})]= (-1)^{s+d_{i+1}}\binom{n-i+1}{s-i-1} c_{\textup{CM}}(\Sigma^{i+1}).$$
\end{proposition}

\section{The Euler characteristic of the stabilization of a determinantal variety}

%\subsection{Integral closure of modules}

The key tool for the next results is the theory of integral closure of modules and multiplicity of pairs of modules. For the convenience of the reader, we  recall the definition in the Appendix. Based on this theory, in this section, we compute the Euler characteristic of the stabilization of an essentially isolated determinantal singularity (Def.\ \ref{EIDS}). The results of this part are mainly based on  \cite{GRa,GR}.

\begin{definition}

%Let $M \in Mat_{(n,n+k)}(\mathbb{C})$ an  $f$
%the map defined by the $t\times t$ minors of $M$. We say that
%$X$ is a determinantal variety if $X$ is defined by the equation $f=0$ and it has codimension $(n-t+1)(n+k-t+1)$.
%\end{definition}

Let $F: U \subset \mathbb{C}^{q} \longrightarrow \operatorname{Hom}(\mathbb{C}^{n},\mathbb{C}^{n+k})$ be an analytic map, where $U$ is an open neighborhood of $0$ and $F(0)=0$. The variety $X=F^{-1}(\Sigma^t)$, $ 1\leq t \leq n$,  in $\mathbb{C}^q$ is called a determinantal variety in 
$U $ of type $(n+k, n, t)$, if  $\operatorname{codim} (X)=  \operatorname{codim} \Sigma^t$, where $\operatorname{codim}$ denotes the codimension of the variety in the ambient space.
\end{definition}

%the determinantal variety in $\mathbb{C}^q$ is the set
%$X=M^{-1}(\Sigma^{t})$, with $1\leq t\leq n,$ 

In \cite{EG} Ebeling and Gusein--Zade introduced the notion of a determinantal variety with an \textit{{essentially isolated determinantal singularity}} (EIDS) (\cite[Section 1]{EG}).

\begin{definition}\label{EIDS}
A determinantal variety $X \subset U$, where $U$ is an open neighborhood of $0$ in $\mathbb{C}^q$, defined by $X=F^{-1}(\Sigma^t)$, $ 1\leq t \leq n$, where $F: U \subset \mathbb{C}^{q} \longrightarrow \operatorname{Hom}(\mathbb{C}^{n},\mathbb{C}^{n+k})$ is an analytic map, is an essentially isolated determinantal singularity (EIDS), if $F$ is transverse to the rank stratification of $\operatorname{Hom}(\mathbb C^n, \mathbb C^{n+k})$  except possibly at the origin. 
\end{definition}

If $X$ is an EIDS in 
$U $ of type $(n+k, n, t)$, the singular
set of $X$ is given by $F^{-1}(\Sigma^{t-1})$.  The regular part
of $X$ is given by $F^{-1}(\Sigma^{t}\setminus \Sigma^{t-1})$ and denoted by
$X_{\text{reg}}$.  As mentioned by Ebeling and Gusein--Zade in their work, an EIDS  $X$ has an isolated singularity at the origin if, and only if, $q \leq (n-t+2) (n+k-t+2)$. 

A  deformation $\tilde F: U \subset  \mathbb C^q   \to \operatorname{Hom}(\mathbb C^n,\mathbb C^{n+k})$ of $F$ which is transverse to the rank stratification is called a \textit{stabilization} of $F$.
According to Thom's Transversality Theorem, $F$ always admits a stabilization $\widetilde F$.  References for the Thom's Transversality Theorem in the analytic setting are \cite{SZ,Tri}, see also the discussion about it in \cite{F}. The refence for the original work of Thom is \cite{Thom}. 

The variety $\widetilde X= \widetilde F^{-1}(\Sigma^t)$ is an
\textit{essential smoothing} of $X$ (\cite[Section 1]{EG}). Ebeling and Gusein--Zade also remarked that, in the specific case that $q < (n-t+2) (n+k-t+2)$ the essential smoothing is a genuine smoothing.
%Moreover, according to the Thom transversality theorem an EIDS always admits an {\it{essential smoothing}} (\cite{EG}, Section $1$),  and in the specific case that %$n<(s-t+2) (p-t+2)$ the {\it{essential smoothing}} is a genuine smoothing. 
% % % % % % % %

In \cite{EG2}, Ebeling and Gusein--Zade studied the radial index and the Euler obstruction of $1$-form on a singular variety. The authors presented a formula expressing the radial index of a $1$-form in terms of the Euler obstructions of the $1$-form on different strata. %This formula contains certain integer coefficients $n_{i t}$. Up to sign, they are equal to the reduced Euler characteristics of generic hyperplane sections of normal slices of the variety at points of different strata of a Whitney stratification.  %For a determinantal singularity $X = F^{-1}(\Sigma^{t})$ outside the origin, these slices are standard ones: a normal slice to the stratum $X_i \setminus X_{i-1}$ is isomorphic to $\Sigma^{t-i+1} \subset Hom{(n-i+1,n+k-i+1)}$. 

% % % % %

%\subsection{The Euler characteristic of the stabilization of an EIDS}

To state the main result of this section, let us work in the following setup.

Let $F$ be an analytic map, $$F: U \subset \mathbb{C}^{q} \longrightarrow \operatorname{Hom}(\mathbb{C}^{n},\mathbb{C}^{n+k}),$$ with $F(0)=0$, where $U$ is an open neighborhood of $0$ in $\mathbb{C}^q$, %$q \geq n(n+k)$, 
defining the EIDS $_sX=F^{-1}(\Sigma^s)$, $ 1\leq s \leq n$. In this set-up, if $q>n(n+k)$, then the closure of the stratum of lowest positive dimension will be the inverse image of the zero matrix, hence will be an ICIS of positive dimension. If  $q=n(n+k)$, then the inverse  image of the zero matrix is still an ICIS but of dimension $0$. This makes it easy to calculate those invariants which are associated to these strata. Away from the origin, the other strata are the same (up to a smooth factor) as the corresponding strata in $\operatorname{Hom}(\mathbb{C}^{n},\mathbb{C}^{n+k})$. This is why we are able to approximate invariants of the singularities associated with $F$ with those of the strata of $\operatorname{Hom}(\mathbb{C}^{n},\mathbb{C}^{n+k})$.

In \cite{GR} the authors work with the multiplicity of the pair $(JM(_{i}X), N(_{i}X))$, where $JM(_{i}X)$ is the Jacobian module and $N(_{i}X)$ is the module of infinitesimal first order deformation of $_{i}X$, induced from the first order deformation of the presentation matrix of $X$. They showed that the multiplicity of that pair is well defined for EIDS. To do that they use the holomorphic triviality of the stratification of $\operatorname{Hom}(\mathbb{C}^{n},\mathbb{C}^{n+k})$. The next result is Corollary 2.16 of \cite{GR}, where  $$n_{i t}= (-1)^{k(t-i)}\binom{n-i}{n-t}.$$ These coefficients will be very important in the  results to follow.

Consider the graph of $F$ in $\mathbb{C}^{q}\times Hom(\mathbb{C}^{n},\mathbb{C}^{n+k})$;  let $F(\mathbb{C}^{q})\cdot P_{d}(\Sigma^{i})$ denote the intersection multiplicity of the graph of $F$ with $\mathbb{C}^{q}\times P_{d}(\Sigma^{i})$, where $P_{d}(\Sigma^{i})$ is the codimension $d$-polar variety of  $\Sigma^{i}$, and $d+$codim$\Sigma^{i}=q$. For a general choice of $P_{d}(\Sigma^{i})$, the graph of $F$ will intersect $\mathbb{C}^{q}\times P_{d}(\Sigma^{i})$ properly, so the intersection multiplicity will be well defined.

\begin{proposition}[{\cite[Cor.\ 2.16]{GR}}] Let $U \subset \mathbb{C}^{q}$ be an open neighborhood at the origin and  $X \subset U$ be an EIDS of type $(n+k,n,t)$, defined by an analytic map $F: U \subset \mathbb{C}^{q} \to \operatorname{Hom}(n+k,n)$, with $F(0)=0$ and $H$ a generic hyperplane through the origin. $\widetilde{X}$ denotes the stabilization of $X$ and $\widetilde{X \cap H}$ denotes the stabilization of ${X \cap H}$. In this setting we have:
\begin{multline*}
(-1)^{\dim X}\chi(\widetilde{X})+(-1)^{\dim X - 1}\chi(\widetilde{X \cap H})=\\
 \sum_{i=1}^{t}n_{it}(e(JM(_{i}X^{d_{i}}), N(_{i}X^{d_{i}}))+ F(\mathbb{C}^{q})\cdot P_{d(i)}(\Sigma^{i})),
\end{multline*} where $d_i= \text{dim}\, _{i}X.$
\end{proposition}

\begin{remark}\label{equations} 

Let $X$ be a determinantal variety of type $(n+k,n,t)$ defined by $F: \mathbb{C}^{q}  \to \operatorname{Hom}(n,n+k)$, where $ q \geq n(n+k)$. Let us denote by $F_{\ell}$ the map $$F_{\ell}: \mathbb{C}^{q} \cap H_{1}\cap \dots \cap H_{\ell} \to \operatorname{Hom}(n,n+k),$$ the restriction of $F$, where $H_{i}$'s are hyperplanes in $\mathbb{C}^{q}$, $F_{0}=F$.
We denote $F_{\ell}^{-1}(\Sigma^{s})$ by $_{s}X_{\ell}$. %where $d(s)=\dim$ $_{s}X$. 
Using the last result for each $\ell$ from $0$ to $d(t)$, and remembering that $X= _{t}X$ we get the equations.

\begin{equation}
\begin{aligned}
&(-1)^{d(t)} \chi( \widetilde{X}) + (-1)^{d(t)-1}\chi(\widetilde{X_{1}})=\\ &\sum\limits_{i=1}^{t} n_{it}(e(JM( _{i}X), N(_{i}X))+ F(\mathbb{C}^{q})\cdot P_{d(i)}(\Sigma^{i}))
\end{aligned}
\tag{0}
\end{equation}
%{ \noindent $(0):  (-1)^{d(t)} \chi( \widetilde{X}) + (-1)^{d(t)-1}\chi(\widetilde{X_{1}})= \sum\limits_{i=1}^{t} n_{it}(e(JM( _{i}X), N(_{i}X))+ F(\mathbb{C}^{q})\cdot \Gamma _{d(i)}(\Sigma^{i})).$}
%
$$\vdots$$
\begin{equation}
\begin{aligned}
&(-1)^{d(t)- \ell} \chi(\widetilde{X_{\ell}})\nota{wide?} + (-1)^{d(t)-\ell-1}\chi(\widetilde{X_{\ell + 1}})= \\
&\sum_{i=1}^{t} n_{it}(e(JM( _{i}X_{\ell}), N(_iX_{\ell}))+ F_{\ell}(\mathbb{C}^{q-\ell})\cdot P_{d(i)-\ell}(\Sigma^{i}) )
\end{aligned}
\tag{$\ell$}
\end{equation}
%
%{\noindent $(\ell): (-1)^{d(t)- \ell} \chi(\widetilde{X_{\ell}}) + (-1)^{d(t)-\ell-1}\chi(  \widetilde{X_{\ell + 1}})= \sum\limits_{i=1}^{t} n_{it}(e(JM( _{i}X_{\ell}), N(_iX_{\ell}))+ F_{\ell}(\mathbb{C}^{q-\ell})\cdot P_{d(i)-\ell}(\Sigma^{i}) ).$}
%
$$\vdots$$
\begin{equation}
\begin{aligned}
(-1)^{0} \chi( \widetilde{X_{d(t)}})  &= n_{t t}( e(JM( _tX_{d(t)}), N(_tX_{d(t)}))\\ &{\quad} + F_{d(t)}(\mathbb{C}^{q-d(t)})\cdot P_{{0}}(\Sigma^{t}))
\end{aligned}
\tag{$d(t)$}
\end{equation}
%{\noindent $(d(t)): (-1)^{0} \chi( \widetilde{X_{d(t)}})  = n_{t t}( e(JM( _tX_{d(t)}), N(_tX_{d(t)})) + F_{d(t)}(\mathbb{C}^{q-d(t)})\cdot P_{{0}}(\Sigma^{t})) .$}
 
 In the last equation the dimension of the variety $X_{d(t)}$ is zero, then the multiplicity of the pair of
 modules is also zero. Then, we can rewrite the equation as
 $$\chi (\tilde{X}_{d(t)}) =F_{d(t)}(\mathbb C^{q-d(t)}) \cdot \Sigma^t$$
 
 When $l=q-n(n+k)$ then $ \chi (_1\tilde{X}_{l}) $ is the number of points of rank $0$ that appears in
 a stabilization of $F_l.$ If $t=1$, then $X$ is an ICIS, and for $l<q-n(n+k)$, the right hand side of the equation becomes $e(JM( _{1}X_{\ell}))$, because $N(_1X_{\ell})$ is free, and the polar varieties of a point are empty, except for $d=0$.
 
\end{remark}

One of the main ingredients to prove our main result of this section is the next proposition. As remarked in \cite{GR}, in the EIDS context, instead of a smoothing, we have a stabilization -- a determinantal deformation of $X$ to the generic fiber. Then the multiplicity of the polar curve of $JM_{z}(\mathcal{X})$ over the parameter space at the origin in a stabilization is the number of critical points that a generic linear form has on the complement of the singular set on a generic fiber. Call this number $m_{d}(X)$, where $d= \dim X$. The $m_{d}$ multiplicity was defined by Gaffney in \cite{G-Top} for the study of isolated complete intersection singularities (ICIS), and for isolated singularities whose versal deformation have a smooth base in \cite{Gaff1}.

\begin{proposition}[{\cite[Prop.\ 2.15]{GR}}] Let $n,k,t \in \mathbb{Z}$, $n>0$, $k \geq 0$ and $0 \leq t \leq n$. Suppose $\mathcal{X}$ is an one parameter stabilization of an EIDS $X$ of type $(n+k,n,t)$, then$$e(JM(X),N(X))+ F(\mathbb{C}^{q})\cdot P_{d}(\Sigma^{s})= m_{d}(X),$$ where  $d= \dim X$. 
\end{proposition}

Now we present a polar multiplicity type formula, based on the connection between multiplicity of polar varieties and the $m_{d}$ invariant. 
%{Given $M$ we have equations for the graph $a_{ij}-m_{ij}$ if we eliminate $m_{ij}$ as much as possible we get generators of the ideal of $M$, $I(M)$ we would like to use this idea to define $e(M_{0}, N_{0})$ work on $\Sigma^{s}$, $N_{0}=m$, $M_{0} = I(M | H_{d})$. Notice that $e(M_{0},N_{0})+ m(\Sigma^{s}(M))$}.

\begin{lemma}\label{lemma} Suppose $X^{d}= F^{-1}(\Sigma^{r})$ is an EIDS, $\pi_{2}\: (X,0) \to (\mathbb{C}^{2},0)$ defines the polar curve $P^{1}(X)$. If $H$ is a generic hyperplane then 
$$m_{d-1}(X\cap H)= m_{d-1}(F\mid_{H})^{-1}(\Sigma^{r} ) = m_0(P^{1}(X)).$$
\end{lemma}

\begin{proof}

It is obvious that $X \cap H =(F\mid_{H})^{-1}(\Sigma^{r})$, for $H$ generic.
Because $X$ is an EIDS, if $H= l^{-1}(0)$ then the deformation $(F\mid_{l^{-1}(t)})^{-1}(\Sigma^{r})$ can be  made by a good choice of $H$ into a stabilization of $X \cap H$. 

By Lemma 4.1.8 of \cite{LT1}, we can pick the kernel of $\pi_{2}$ so that the intersection of the tangent cone of the polar with the kernel of the projection is $\{0\}.$ We can assume $\pi_{2}(z)=(l_{1}(z),l_{2}(z))$ such that $\deg(l_{1}\mid P^{1}(X))= m_0(P^{1}(X))$, where $l_{1}^{-1}(0)=H$, then the critical points of $l_{2}$ over the fibers of $l_{1}$ is $P^{1}(X)$ and the polar curve of the deformation of $X\cap H$ to its stabilization. This  implies that the multiplicity of the polar curve $m_0(P^{1}(X))=m_{d-1}(X\cap H)$.
\end{proof}

The last result depends heavily on the landscape of the computations remaining the same throughout the proof, as the following example shows.

\begin{exmp}

Let  $X=F(\mathbb{C}^2)$, where
$F(x,y)=(x, y^2, xy)$. We see that  $X $is the Whitney umbrella in $\mathbb{C}^3$. We consider the landscape of sections of $X$ by maps $f\: \mathbb{C}^2\to \mathbb{C}^3$, and their deformations. 

If $H$, a generic hyperplane in $\mathbb{C}^3$, is the image of $f$, then $X\cap H$ is a cusp. The polar curve of $X$ is smooth, and therefore $m_1(f^{-1}(X))$ is $1$ in this landscape by the above proof.

If we change the landscape, we get a different result. Using formulas from Section 8 of  \cite{G-Top} we compute easily that $m_1(X)=m_0(X)+m_2(X)-1=2+0-1=1$, as expected,  and $m_1(f^{-1}(X))=\mu(X\cap H)+m_0(X\cap H)-1=2+2-1=3$, in the landscape of ICIS. 

\end{exmp}

In the above example, in the landscape of sections $f^{-1}(X)$ has no smoothing; the stabilization has a node. In the ICIS landscape it does. The node accounts for the difference in the two values of $m_1(f^{-1}(X))$.

Now we can state the main result about the Euler characteristic of a stabilization.

\begin{theorem}\label{stabilization} Let $n,k,t,q \in \mathbb{Z}$, $n>0$, $k \geq 0$ and $0 \leq t \leq n$. Let us assume $ q \geq n(n+k)$ and $X \subset \mathbb{C}^{q}$ be an EIDS of type $(n+k,n,t)$ given by an analytic map $F\: \mathbb{C}^{q}\to \operatorname{Hom}(n,n+k)$, with $F(0)=0$. Then,
\begin{multline*}
\chi(\widetilde{X})\nota{wide?} = (-1)^{d(t)-d(1)}n_{1 t} \chi(\widetilde{_{1}X})+\\ \sum_{i=2}^{t} n_{i t}((-1)^{d(t)}m_{d(i)}(_{i}X)+ (-1)^{d(t)-d(i)}\operatorname{Eu}_{0}(_{i}X)),$$
\end{multline*}where $$n_{i t}= (-1)^{k(t-i)}\binom{n-i}{n-t}$$ and $d(i)= \dim {_{i}X}$ for $i= 0, \dots, t$.
\end{theorem}

\begin{proof} Taking the alternating sum of the equations of Remark \ref{equations} and multiplying both sides by $(-1)^{d(t)}$,  it is easy to see that the left hand side is $\chi(\widetilde{X})$.

On the right hand side  the alternating sum is$$\sum_{\ell =0}^{d(t)} (-1)^{\ell} \sum_{i=1}^{t} n_{it}(e(JM( _{i}X_{\ell}), N(_iX_{\ell}))+ F_{\ell}(\mathbb{C}^{q-\ell})\cdot P_{d(i)-\ell}(\Sigma^{i}) ).$$ 

Here the alternating sum requires a more delicate analysis. So we will split the sum in two alternating sums. 

The first one is when $i=1$, in this case we have that all $_{1}X_{\ell}$ are ICIS, and in this case we know that, $$m_{0}(_{1}X)-m_{1}(_{1}X)+\dots + (-1)^{d(1)} m_{d(1)}(_{1}X) = \chi(\widetilde{_{1}X})\nota{wide?}$$
(see for instance Remark 4.6 of \cite{Victor-Marcelo}).

Now, as the hyperplanes $H_{1}, \dots, H_{\ell}$ are generic, applying successively the Lemma \ref{lemma}  we have the following relation,
$$m_{d(j)-\ell}(_{j}X_{\ell})= m_{d(j)-\ell}(_{j}X),$$for $1 < \ell \leq d(j)$. %\cite[Cor.\ 2.3.2.1]{LT2} (see also \cite[Prop.\ 5.4.2]{Te}).

Using this last relation and Theorem \ref{LT},  the second part of the alternating sum, which contains those terms with $i>1$, can be rewritten as 
$$\sum_{i=2}^{t} n_{i t}((-1)^{d(t)}m_{d(i)}(_{i}X)+ (-1)^{d(t)-d(i)}\operatorname{Eu}_{0}(_{i}X)).$$

So we have:
\begin{multline*}
\chi(\widetilde{X})= (-1)^{d(t)-d(1)}n_{1 t} \chi(\widetilde{_{1}X})+\\ \sum_{i=2}^{t} n_{i t}((-1)^{d(t)}m_{d(i)}(_{i}X)+ (-1)^{d(t)-d(i)}\operatorname{Eu}_{0}(_{i}X)).
\end{multline*}
\end{proof}

\begin{remark}
	
%To simplify we use the notation $M_{n+k,n}^i=M^i, \, 1\leq i\leq s\leq n,$ and %$d_i=\text{dim} (M^i),$ and $_iX=F^{-1}(M^i),$ and $d_i^{'}==\text{dim} (_iX).$

In \cite{GR}, applying $Prop.\ 4$ of \cite{EG} to EIDS case, Gaffney and Ruas obtained the equations below. Recall that the integers $n_{is}$ are given by the formulas 
$n_{is}=
 (-1)^{(k)(s-i)} \binom{n-i}{n-s}$. 
 %Further, if $\omega=dl$, $l$ a generic linear form, then $\indrad (dl,X,0) =(-1)^{d-1}\overline {\chi}( \widetilde{X\cap H}),$ $H=l^{-1}(\epsilon)$ for $\epsilon$ %sufficiently small  (\cite{E-GZ1}, Theorem 3). This invariant does not depend on the generic hyperplane $H$ or on the
 %stabilization of the determinantal section $X\cap H,$ so we simply write $\indrad (dl,X,0) =(-1)^{d-1}\overline {\chi}( {X\cap H}).$

%Assuming $X= {}_tX$, we can re-write this formula using $dl$ for $\omega$ and $m_{d_i}(_iX)$ for $ \indPHN (dl,  \hskip 1pt _iX,0) $, apply it to each stratum in %turn getting:

\begin{align*}
 (-1)^{d_s}{\chi}(_sX)+(-1)^{d_s-1}{\chi}( {_sX \cap H}) &= \sum_{i=1}^s n_{is} m_{{d_i}}(_iX) \\
 (-1)^{d_{s-1}} {\chi}( _{s-1}X)+(-1)^{d'_{s-1}-1}{\chi}({ _{s-1}X \cap H}) &= \sum_{i=1}^{s-1} n_{i(s-1)}  m_{d_i}(_iX) \\
& \vdots \\
 (-1)^{d_2 }{\chi}( _2X)+(-1)^{d_2-1}{\chi}( {_2X  \cap H})  &=\sum_{i=1}^{2} n_{i(2)} m_{d_i}(_iX) \\
(-1)^{d_1}\chi(_1X)+ (-1)^{d_1-1}\chi(_1X\cap H)&= m_{d_1}(_1X)
\end{align*}

We consider the above equations forming a system of $s$ equations and $s$ variables $m_{{d_i} }( _iX)$, $i=1, \dots, s$.

The $s \times s$  matrix of the system is the triangular matrix:
\[ A =
\begin{bmatrix}
n_{1s} & n_{2s} & \cdots & \cdots &n_{ss} \\
n_{1s-1} & n_{2s-1} & \cdots & n_{s-1 s-1}& 0\\
& & & & \\
& & & & \\
n_{12} & n_{22}&\cdots & 0 & 0\\
n_{11} & \cdots & 0 & 0 & 0\\
\end{bmatrix}.
\]
%\left[
%\begin{array}{lcrll}
%n_{1s} & n_{2s} & \ldots & \ldots &n_{ss} \\
%n_{1s-1} & n_{2s-1} & \ldots & n_{s-1 s-1}& 0\\
%& & & & \\
%& & & & \\
%
%
%n_{12} & n_{22}&\ldots & 0 & 0\\
%n_{11} & \ldots & 0 & 0 & 0\\
%\end{array}
%\right],

Since $n_{i i}= 1$, it follows that  $\det A=(-1)^s$.

We can rewrite the above system as follows:
$$ A\cdot  X= B.$$ %where $B$ is the $s\times 1$ matrix given by

Write  $x_i$, $i=1,\dots, s$ and $b_i$, $i=1,\dots s$ the entries of the matrices $X$ and $B$ respectively.  Then, $b_i= (-1)^{d_i }{\chi}(_iX)+(-1)^{d_i-1}{\chi}( {_iX \cap H})$, $i=2, \dots, s$ and $b_1=(-1)^{d_1}\chi(_1X)+ (-1)^{d_1-1}\chi(_1X\cap H)$.  And $x_i=m_{d_i}(_iX)$.

The solution of the above system are given by:
\[X:=
\left\{
\begin{aligned}
x_1 &= b_1\\
x_2 &= b_2- n_{12} x_1\\
%(-1)^{d_s-1+s} &\begin{pmatrix}n-1\\ i-1\end{pmatrix} \\
 &{\ }\vdots\\
x_s &= b_s- \Sigma_{i=1}^{s-1} n_{is}x_i
\end{aligned} 
\right.
\]
\end{remark}

\begin{proposition}\label{cidinha} Let $\Sigma^{i} \subset H(n,n+k)$ be the generic determinantal variety and $i=1,\dots, n$, then,
$$m_{0}(\Sigma^1)=1, \quad  m_{d'_i}(\Sigma^i)=0\  (2\leq i\leq n),$$ where $d'_{i}= \dim \Sigma^{i}$.
\end{proposition}

\begin{proof}
$\Sigma^1=\{0\}$ is a $0$-dimensional complete intersection of multiplicity $1$,
so $x_1=m_0(\{0\})=1$. 

We first solve the above system for the generic determinantal varieties 
$\Sigma^i$, $ 1\leq i\leq n$, we denote $d'_{i}= \dim \Sigma^{i}$.

First, notice that the following hold:
\begin{itemize}
	
	\item [ (1)]  For all $i$,   $\chi(\Sigma^i)=1$. This holds because $\Sigma^i$ is its own
	essential smoothing. So we are computing the Euler characteristic of a cone with vertex at the origin, hence its topology is the topology of a point.

	\item [(2)] $ \chi(\Sigma^i \cap H) = 1+ \bar{\chi}{(\Sigma^i\cap H )}= 1+(-1)^i \begin{pmatrix}n-1\\ i-1\end{pmatrix}$  (\cite[Prop.\ 3]{EG}).
\end{itemize}

Then, the matrix $B $ is as follows:
\[ B =
\begin{bmatrix}
(-1)^{d'_s-1+s} & \binom{n-1}{s-1}\\[1ex]
(-1)^{d'_{s-1}-1+s-1} & \binom{n-1}{s-2} \\[1ex]
%(-1)^{d_s-1+s} &\begin{pmatrix}n-1\\ i-1\end{pmatrix} \\
%
\vdots & \vdots \\
(-1)^{d'_2-1+2} & \binom{n-1}{1}\\[1ex]
\nota{????} & 1
\end{bmatrix},
\]
and $X$ is the matrix of the variables:
\[ X =
\begin{bmatrix}
m_{d'_1}(\Sigma^1)\\[1ex]
m_{d'_2}(\Sigma^2) \\[1ex]
\vdots\\
m_{d'_s}(\Sigma^s)
\end{bmatrix}.
\]

Notice that $\Sigma^1=\{0\}$ and $d'_1=0$, hence $ x_1=1$.

%The polar multiplicities  will be given by $X= A^{-1}B.$

Now, $$x_2= (-1)^{d'_2-1+2}  \begin{pmatrix}n-1\\ 1\end{pmatrix}- (-1)^{k} \begin{pmatrix}n-1\\ n-2\end{pmatrix},$$ and we can verify that $d'_2$ and $k+1$ have the same parity. Hence $m_{d'_2}(\Sigma^2)=0$. The result now follows by induction.
\end{proof}

In the next set of results, we  work with a determinantal variety $X$, satisfying some specific conditions that we will call nice conditions. Since these conditions ensure that $X$ is similar to the generic determinantal varieties, and since an EIDS arises as a pullback of a generic determintal variety, we will call an $X$ satisfying these conditions a {\it good approximation}.

\begin{definition} Let $X \subset \mathbb{C}^{q}$ be an EIDS of type $(n+k,n,t)$. We say that $X$ satisfies the nice condition if $$e(JM(_{i}X_{\ell}), N(_{i}X_{\ell}))=0,$$for $\ell < d(i)$, $1< i \leq t$, where $ _tX,\hskip 1pt _tX_{\ell}$ are EIDS, and  $ _1X, \hskip 1pt _1X_{\ell}$ are ICIS. %with

%$$e(JM(_1 X_{\ell}), N(_1 X_{\ell}))= e(JM(_1 X_{\ell}))= \mu(_1 X_{\ell})+ \mu(_1 X_{\ell + 1}).$$

We also assume $F_{\ell}(\mathbb{C}^{q-\ell}) \cdot P_{q-l}(\Sigma^{i})= m_0(P_{q-\ell}(\Sigma^{i}))$ here $q-\ell$ is codimension in $\operatorname{Hom}(n,n+k)$, and if $P_{q-\ell}(\Sigma^{i})= \emptyset$ the multiplicity is $0$.
\end{definition}

For this special situation, using the Theorem \ref{EulerFormula}, we get a very nice formula in terms of Newton's binomials.

\begin{proposition}\label{stabilizationEIDS} Let $X$ be an EIDS of type $(n+k,n,t)$ which is a good approximation, given by $F: \mathbb{C}^{q} \to \operatorname{Hom}(n, n+k)$. Then,$$\chi(\widetilde{X})\nota{wide?}= (-1)^{d(t)-d(1)}n_{1t}\chi ( _1 \widetilde{X}) + \sum_{i=2}^{t} (-1)^{d(t)-d(i)}n_{i t}\binom{n}{i-1},$$where$$n_{i t}= (-1)^{k(t-i)}\binom{n-i}{n-t}$$where $d(i)= \dim {_iX}$ for $i= 0, \dots, t$.
\end{proposition}

\begin{proof}

If $X$ satisfies the nice condition, then $e(JM(_{i}X_{\ell}), N(_{i}X))=0$, hence $m_{d(i)-\ell}(_{i}X_{\ell})= F(\mathbb{C}^{q-\ell})\cdot P_{d(i)-\ell}(\Sigma^{i})$.

But $F(\mathbb{C}^{q-\ell})\cdot P_{d(i)-\ell}(\Sigma^{i})= m_{d(i)-\ell}(\Sigma^{i})$
when $d(i)-\ell \leq d'(i)= \dim(\Sigma^{i})$.

When $d(i)- \ell > d'(i)$, the polar is empty, the polar multiplicity is zero.

Recall that, by Proposition \ref{cidinha}, the top polar multiplicity of the generic determinantal variety $m_{d'(i)}(\Sigma^{i})$ is also zero.

Hence, since $X$ is a good approximation, it satisfies the nice condition and using the formula of the Theorem \ref{EulerFormula} to compute the Euler obstruction we finish the proof.
\end{proof}

\begin{remark} Some interesting special cases are:
\begin{itemize}

\item[(1)] If $n=2, t=2$ then $X$ is defined by maximal minors. $$(-1)^{d(2)-d(1)}=(-1)^{(q-(k+1))-q+2(k+2)}=(-1)^{k+1},$$ and therefore $$\chi( \widetilde{X})=(-1)^{k+1}n_{1 2}\chi( _1 \widetilde{X})+\binom{2}{1}= 2 - \chi(_1\widetilde{X}).$$
And using the reduced Euler characteristic we have
$$\overline{\chi}(\widetilde{X})= -\overline{\chi}(_1 \widetilde{X})= (-1)^{q-(4+2k)}(-1)\mu(_1 X)=(-1)^{q-3}\mu(_1X).$$

\item[(2)] If $n \in \mathbb{Z}, n \geq 2, t=2$ then $X$ is defined by $2 \times 2$ minors. The algebraic variety $\Sigma^{2} \subset \operatorname{Hom}(n,n+k)$ has isolated singularity at $0$, so $$e(2,n)= \chi(\Sigma^{2}\cap H_{t}),$$where $H_{t}$ is parallel to generic hyperplane through the origin.

By Ebeling and Gusein--Zade this is $$1+ (-1)^{2} \binom{n-1}{2-1}= 1+n -1=n,$$thus $$\chi( \widetilde{X})= -\chi(_1 \widetilde{X})+n.$$
\end{itemize}

\end{remark}

%Note that for the singularities we consider here we can always chose for each $t, n$ and $k$ such that $Eu_{_t X}(0) = e(t,n)$.

%Applying the formula of the Theorem \ref{EulerFormula} for the Euler obstruction to the earlier formula of Theorem \ref{stabilization} for the Euler characteristic of a stabilization of a generic determinantal variety we have the following result.

 %\begin{proposition}Suppose that $_tX$ is an EIDS that satisfies the nice condition, then,$$\chi(_t\widetilde{X})= (-1)^{\dim \Sigma^{t}+k(t-1)} ( \chi(_1 \widetilde{X})) + \sum_{s=2}^{t} (-1)^{k(t-s)+d(t)-d(s)}\binom{n-s}{n-t}\binom{n}{s-1}.$$\end{proposition}

Now, using (2), from the proof of Theorem \ref{EulerFormula}, and using that if $X$ is an EIDS of type $(n+k,n,s)$, defined by the analytic map $F\: U \subset \mathbb{C}^{q} \to \operatorname{Hom} (n,n+k)$, when $q> n(n+k)$ we have a submersion on the strata different from $\{0\}$, so we have a fibered structure and when $q < n(n+k)$ we have an immersion, using the transversality of $F$ we have these two next results.

\begin{proposition} Let $X \subset \mathbb{C}^{q}$ be an EIDS of type $(n+k,n,s)$, defined by the analytic map $F\: U \subset \mathbb{C}^{q} \to \operatorname{Hom} (n,n+k)$, with $F(0)=0$ and $n(n+k)>q$. In this setting we have
$$\operatorname{Eu}_{0}(X)= e(s-1,n-1) + \sum_{i=2}^{s} \overline{\chi}_{*}(i,n)\binom{n-i}{s-i},$$where $\overline{\chi}_{*}(i,n)= \overline{\chi}(_{i}X \cap l^{-1}(t_{0})\cap B_{\varepsilon}(0))$ and $l$ a generic linear form.
\end{proposition}

%\begin{remark}
%This amounts to a stabilization of $_{i} X \cap l^{-1}(0) \cap B_{\varepsilon}(0)$; it is %possible to compute the invariant using material about stabilizations, most completely %for the maximal minor singularities.

%Note if $n=2$ this reduces to $1+\overline{\chi}_{*}(2,2)= {\chi}_{*}(2,2)$, as it should.

%\end{remark}

\begin{proposition} Let $X \subset \mathbb{C}^{q}$ be an EIDS of type $(n+k,n,s)$, defined by the analytic map $F\: U \subset \mathbb{C}^{q} \to \operatorname{Hom} (n,n+k)$, with $F(0)=0$ and $q > n(n+k)$. In this setting we have
$$\operatorname{Eu}_{0}(X)= \binom{n}{s-1} + \overline{\chi}(_{1}X \cap H) \binom{n-1}{s-1} + \sum\limits_{i=2}^{s} \overline{\chi_{*}}(i,n) \binom{n-1}{s-1}.$$
\end{proposition}

\begin{remark}
Note that in this case, when  $q > n(n+k)$ there are $2$ additional terms. Also note that $_{1}X$ is always an ICIS so $\overline{\chi}(_{1}X \cap H)$ is $\mu (_{1}X \cap H)$ up to a sign. Note also that if $n=2$ we get$$\operatorname{Eu}_{0}(X)= 2 + (-1)^{q-1}\mu (_{1}X \cap H) + \overline{\chi}(_{2}X \cap H),$$ which is Siesqu\'{e}n's formula \cite{Nancy}.
\end{remark}

Generalizing Definition \ref{polar}, we need now the notion of polar varieties of a module $M$.

\begin{definition}\label{PolarGen}
Given a submodule $M$ of a free
$\mathcal{O}_{X^{d}}$ module $F$ of rank $p$ (in our case $F = \mathcal{O}^{p}_{X^{d}} $). Let us say the generic rank of $M$ is $g$. The \textit{polar variety of codimension $k$} of $M$ in $X$, denoted $P_k(M)$, is constructed by intersecting $\Projan{\mathcal R}(M)$
 with $X\times H_{g+k-1}$ where $H_{g+k-1}$ is a general plane of codimension $g+k-1$, then projecting to
$X$.
\end{definition}

Let us define the local Euler obstruction of a module $M \subset \mathcal{O}_{X^{d}}^{p}$.

\begin{definition}\label{Genob}
Assume $X$ equidimensional, generically reduced. Given a sheaf of modules $M \subset \mathcal{O}_{X^{d}}^{p}$, $M$ with the same generic rank on each component of $X$.
We define $$\operatorname{Eu}_{0}(M)= \sum\limits_{i=0}^{d-1}(-1)^{i}m_0(P_{i}(M)),$$where $P_{i}(M)$ is the polar variety of $M$ of codimension $i$. Since $X$ is generically reduced, $P_{0}(M)=X$.
\end{definition}

\begin{remark}
When $M= JM(X)$ the generalization of the polar varieties $(Def.\ \ref{PolarGen})$ coincides with the classical notion of polar varieties used by L\^{e} and Teissier to prove Theorem \ref{LT}. In other words, in this case we have $\operatorname{Eu}_{0}(JM(X))= \operatorname{Eu}_{0}(X)$.

\end{remark}

\begin{theorem}\label{theo20}
Given $F\: \mathbb{C}^{q}\to \operatorname{Hom}(n,n+k)$, with $0 <q \leq n(n+k),$ such that $F$ defines a EIDS $X$. Let $M_{i}$,  $i>0$ be defined as $$JM((F\mid_{H^{c(r)+i}})^{-1}(\Sigma^{r}))= JM(F^{-1}(\Sigma^{r})\cap H^{c(r)+i}).$$

Here $H^{c(r)+i}$ is a generic plane of dimension $c(r)+i$, where $c(r)$ is the codimension of $\Sigma^{r}$ in $\operatorname{Hom}(n,n+k)$. Let $N_{i}= (F\mid_{H^{c(r)+i}})^{*}(JM(\Sigma^{r}))$. We let $M_{0}, N_{0}= 0$. Then,
$$\operatorname{Eu}_{0}(JM(X))= \operatorname{Eu}_{0}(X)=  \sum\limits_{i=0}^{d-1} (-1)^{i} e(M_{i}, N_{i}; \mathcal{O}_{X\cap H^{c(r)+i}})  + \operatorname{Eu}_{0}(F^{*}(JM(\Sigma^{r}))).$$
\end{theorem}

\begin{proof}

By definition we have
\begin{align}
\operatorname{Eu}_{0}(JM(X)) &= \sum\limits_{i=0}^{d-1}(-1)^{i} m_0(P_{i}(X)) \nonumber \\
&= \sum\limits_{i=0}^{d-1} (-1)^{-1} m_0(P_{i}(X)\cap H_{1}\cap \dots \cap H_{l}), \tag{$\ast$}
\end{align}
%$$\operatorname{Eu}_{0}(JM(X))= \sum\limits_{i=0}^{d-1}(-1)^{i} m(P_{i}(X)) =$$
%$$(*) = \sum\limits_{i=0}^{d-1} (-1)^{-1} m(P_{i}(X)\cap H_{1}\cap \cdots \cap H_{l}),$$
where $l= q-c(r)-i-1$.

Note that since $P_{i}(X)$ has dimension $q-c(r)-i$ the intersection has dimension 1. As the hyperplanes $H_{1}, \dots, H_{\ell}$ are generic, by  \cite[Cor.\ 2.3.2.1]{LT2} (see also \cite[Prop.\ 5.4.2]{Te}) we have

$$\text{Equation ($\ast$)} = \sum\limits_{i=0}^{d-1}(-1)^{i}m_0(P^{1}(X \cap H^{c(r)+i+1})) $$ by the genericity of polar varieties this is $\sum\limits_{i=0}^{d-1}(-1)^{i}m_{i}(X \cap H^{c(r)+i})$.

If $i=0$ then $m_{0}(X \cap H^{c(r)})$ is the multiplicity of $X$.

$\nota{??}=  \sum\limits_{i=0}^{d-1} (-1)^{i} e(M_{i},N_{i}) + \sum\limits_{i=0}^{d-1}\operatorname{Im}(F\mid_{H^{c(r)+i}}) \cdot P_{i}(\Sigma^{r})$.

But,
\begin{multline*}
\operatorname{Im}(F\mid_{H^{c(r)+i}}) \cdot P_{i}(\Sigma^{r}) = m_0(P^{1}(M\mid_{H^{c(r)+i+1}})^{*}(JM(\Sigma^{r})))=\\
=m_0(P_{1}(M^{*}(JM(\Sigma^{r})))\cap H_{1}\cap H_{2} \cap \dots \cap H_{l})=\\
= m_0(P_{1}(M^{*}(JM(\Sigma^{r}))))
\end{multline*}
where $l=q-c(r)-i-1$.

The result follows.
\end{proof}

This is a very nice formula because it gives the correction term for the difference between the two Euler obstructions in terms of multiplicity of pair of modules.

For the following Corollary, it is convenient to change our notation a little since the main tool is based on \cite{GRa}, so we match the notation there. We let ${\overline \Sigma}_r$ denote $\Sigma^{r+1},$ that is we let ${ \Sigma}_r$ denote the matrices of kernel rank $r.$

\begin{corollary}\label{cor21} Suppose that $X\subset {\mathbb C}^q$ and its generic plane sections are good approximations to  ${\overline \Sigma}_r\subset Hom(n,n+k)$. Suppose that $n(n+k)>q> \dim({\Sigma}^r)$. Then $Eu_0(X)=Eu_0({\overline \Sigma}_r)$.
\end{corollary}
\begin{proof} The hypotheses, together with the last theorem imply that  $Eu_0(X)= Eu_0(F^*(JM({\overline \Sigma}_r)))$, and that the multiplicities of the polar varieties of $F^*(JM({\overline \Sigma}_r))$ agree with those of the corresponding polar varieties of ${\overline \Sigma}_r$. It remains to show that the dimension condition on $ {\mathbb C}^q$ implies that $X$ and ${\overline \Sigma}_r$ have the same number of non-empty polar varieties. 

From \cite{GRa},  the condition that $\Gamma^u({\overline \Sigma}_r)$ be non-empty is:

$$u\ge (n-r)(n+k-r).$$
Phrased in terms of the codimension $c$, this is:

$$c= \dim({\overline \Sigma}_r)-\dim (\Gamma^u({\overline \Sigma}_r))\le \dim ({\overline \Sigma}_r)-(n-r)(n+k-r)$$
$$=\dim ({\overline \Sigma}_r)- \operatorname{codim} ({\overline \Sigma}^r).$$

So we want $\dim ({\overline \Sigma}_r)-\operatorname{codim} ( {\overline \Sigma}^r)< \dim(X)=q-\operatorname{codim}({\overline \Sigma}_r).$
This is equivalent to:
$$\dim({\overline \Sigma}_r)-[(n)(n+k)-\dim ({\overline \Sigma}^r)]<q-[(n)(n+k)-\dim ({\overline \Sigma}_r)]$$
and 
$$\dim ({\overline \Sigma}^r)< q,$$
which is the hypothesis on $q$.

\end{proof}

\begin{exmp} Suppose $r=1$, $k=1$. If $n=2$, then $6>q>dim {\overline \Sigma}^1=\dim {\overline \Sigma}_1=4$, so $q=5$ is the only value that fits. However, 
for general $n$ we get $n(n+1)>q>\dim {\overline \Sigma}^1=2n$, so the number of possible values of $q$ grows quadratically.
\end{exmp}

\section*{Appendix. Multiplicity of pairs of modules}

%\subsection{Integral closure of modules}

Let  $(X, x)$ be the germ of a complex analytic space, $\dim X=d$, and $X$ a
small representative of the germ and let $\mathcal{O}_{X}$ denote the
structure sheaf on a complex analytic space $X$.  The key tool in the work is the theory of integral closure of modules, which we now introduce.

\begin{definition} Suppose $(X, x)$ is the germ of a complex analytic space,
$M$ a submodule of $\mathcal{O}_{X,x}^{p}$. Then $h \in
\mathcal{O}_{X,x}^{p}$ is in the integral closure of $M$, denoted
$\overline{M}$, if for all analytic $\phi : (\mathbb{C}, 0) \to (X,
x)$, $h \circ \phi \in (\phi^{*}M)\mathcal{O}_{1}$. If $M$ is a
submodule of $N$ and $\overline{M} = \overline{N}$ we say that $M$
is a reduction of $N$.
\end{definition}

To check the definition it suffices to check along a finite number of curves whose generic point is in the Zariski open subset of $X$ along which $M$ has maximal rank (cf.\ \cite{G-2}).

If a module $M$ has finite colength in $\mathcal{O}_{X,x}^{p}$, it
is possible to attach a number to the module, its Buchsbaum--Rim
multiplicity,  $e(M,\mathcal{O}_{X,x}^{p})$. We can also define the multiplicity $e(M,N)$ of a pair of
modules $M \subset N$, $M$ of finite colength in $N$, as well, even
if $N$ does not have finite colength in $\mathcal{O}_{X}^{p}$.

We recall how to construct the multiplicity of a pair of modules using the approach of
Kleiman and Thorup \cite{KT}. Given a submodule $M$ of a free
$\mathcal{O}_{X^{d}}$ module $F$ of rank $p$ (in our case $F = \mathcal{O}^{p}_{X^{d}} $), we can associate a
subalgebra $\mathcal{R}(M)$ of the symmetric $\mathcal{O}_{X^{d}}$
algebra on $p$ generators. This is known as the Rees algebra of $M$.
If $(m_{1}, \dots ,m_{p})$ is an element of $M$ then $\sum
m_{i}T_{i}$ is the corresponding element of $\mathcal{R}(M)$. Then
$\Projan \mathcal{R}(M)$, the projective analytic spectrum of
$\mathcal{R}(M)$ is the closure of the projectivised row spaces of
$M$ at points where the rank of a matrix of generators of $M$ is
maximal. Denote the projection to $X^{d}$ by $c$. If $M$ is a
submodule of $N$ or $h$ is a section of $N$, then $h$ and $M$
generate ideals on $\Projan \mathcal{R}(N)$; denote them by $\rho(h)$
and $\rho(\mathcal{M})$. If we can express $h$ in terms of a set of
generators $\{n_{i}\}$ of $N$ as $\sum g_{i}n_{i}$, then in the
chart in which $T_{1}\neq 0$, we can express a generator of
$\rho(h)$ by $\sum g_{i}T_{i}/T_{1}$. Having defined the ideal sheaf
$\rho(\mathcal{M})$, we blow it up.

On the blow up $B_{\rho(\mathcal{M})}(\Projan \mathcal{R}(N))$ we have
two tautological bundles. One is the pullback of the bundle on
$\Projan \mathcal{R}(N)$. The other comes from $\Projan
\mathcal{R}(M)$. Denote the corresponding Chern classes by $c_{M}$
and $c_{N}$, and denote the exceptional divisor by $D_{M,N}$.
Suppose the generic rank of $N$ (and hence of $M$) is $g$.

Then the multiplicity of a pair of modules $M, N$ is:
$$
e(M,N) = \sum_{j=0}^{d+g-2}\int D_{M,N}\cdot c_{M}^{d+g-2-j}\cdot
c_{N}^{j}.
$$

Kleiman and Thorup show that this multiplicity is well defined at $x
\in X$ as long as $\overline{M} = \overline{N}$ on a deleted
neighborhood of $x$. This condition implies that $D_{M,N}$ lies in
the fiber over $x$, hence is compact. Notice that when $N=F$ and $M$ has finite colength in $F$ then $e(M,N)$ is the Buchsbaum--Rim multiplicity $e(M,\mathcal{O}_{X,x}^{p})$. 

There is a fundamental result due to Kleiman and Thorup, the principle of additivity \cite{KT}, which states that given a sequence of $\mathcal{O}_{X,x}$-modules $M\subset N \subset P$  such that the multiplicity of the pairs is well defined, then$$e(M,P)=e(M,N)+e(N,P).$$Also if $\overline{M}=\overline{N}$ then $e(M,N)=0$ and the converse also holds if $X$ is equidimensional. Combining these two results we get that if $\overline{M}=\overline{N}$ then $e(M,P)=e(N,P)$.

In studying the geometry of singular spaces, it is natural to study
pairs of modules. In dealing with non-isolated singularities, the
modules that describe the geometry have non-finite colength, so
their multiplicity is not defined. Instead, it is possible to define
a decreasing sequence of modules, each with finite colength inside
its predecessor, when restricted to a suitable complementary plane.
Each pair controls the geometry in a particular codimension.

As mentioned before, we need now the notion of the polar varieties of $M$. The \textit{polar variety of codimension $k$} of $M$ in $X$, denoted
$P_k(M)$, is constructed by intersecting $\Projan{\mathcal R}(M)$
 with $X\times H_{g+k-1}$ where
$H_{g+k-1}$ is a general plane of codimension $g+k-1$, then projecting to
$X$.

\medskip

\paragraph*{Setup:} We suppose we have families  of modules $M\subset  N$, $M$ and $N$
 submodules of a free module $F$ of rank $p$
 on an equidimensional family of spaces with equidimensional
 fibers ${\mathcal X}^{d+k}$, ${\mathcal X}$ a family over a smooth base
$Y^k$. We assume that the generic rank of $M$, $N$ is $g \le p$.  Let
$P(M)$ denote $\Projan {\mathcal R}(M)$, $\pi_M$
 the projection to ${\mathcal X}$.

We will be interested in computing, as we move from the special point $0$ to a generic point, the change in the multiplicity of
the pair $(M,N)$, denoted $\Delta(e(M,N))$. We will assume that the integral closures of $M$ and $N$ agree off a set $C$ of dimension
$k$ which is finite over $Y$, and assume we are working on a
sufficiently small neighborhood of the origin, so that every component
of $C$ contains the origin in its closure. Then $e(M,N, y)$ is the
sum of the multiplicities of the pair at all points in the fiber of
$C$ over $y$, and $\Delta(e(M,N))$ is the change in this number from
$0$ to a generic value of $y$. If we have a set $S$ which is finite
over $Y$, then we can project $S$ to $Y$, and the degree of the
branched cover at $0$ is $\operatorname{mult}_{y} S$ (of course, this is just the
number of points in the fiber of $S$ over our generic $y$).

Let $C(M)$ denote the locus of points where $M$ is not free, \textit{i.e.}, the
points where the rank of $M$ is less
than $g$, $C(\Projan {\mathcal R}(M))$
its inverse image under $\pi_M$.

We can now state the Multiplicity Polar Theorem. The proof in the ideal case appears in \cite{Gaff1}; the general proof appears in \cite{Gaff}.

\begin{theorem}[Multiplicity Polar Theorem]\label{MPT} Suppose in the above
setup we have that $\overline{M} = \overline{N}$ off a set $C$ of
dimension $k$ which is finite over $Y$. Suppose further that
$C(\Projan\mathcal{R}(M))(0) = C(\Projan\mathcal{R}(M(0)))$ except
possibly at the points which project to $0 \in \mathcal{X}(0)$.
Then, for $y$ a generic point of $Y$, $$\Delta(e(M,N)) =
\operatorname{mult}_{y}P_{d}(M) - \operatorname{mult}_{y}P_{d}(N)$$
where ${\mathcal X}(0)$ is the fiber over $0$ of the family ${\mathcal X}^{d+k}$, $C(\Projan\mathcal{R}(M))(0)$ is the fiber of $C(\Projan\mathcal{R}(M))$ over $0$ and $M(0)$ is the restriction of the module $M$ to  ${\mathcal X}(0)$. \end{theorem}

%\section*{Acknowledgments}

%The authors are grateful to Jonathan Mboyo Esole, Thiago de Melo and Xiping Zhang for their careful reading and for suggesting improvements in this work.

\end{document}